\documentclass[a4paper,11pt]{amsart}

\usepackage{amscd}
\usepackage{xypic}  
\usepackage{amssymb}
\usepackage{amsthm}
\usepackage{epsfig}
\usepackage{paralist}

\usepackage{comment}
\usepackage[T1]{fontenc}
\newtheorem{thm}{Theorem}[section]
\newtheorem{cor}[thm]{Corollary}
\newtheorem{lemma}[thm]{Lemma}
\newtheorem{prop}[thm]{Proposition}

\newtheorem{proposition}[thm]{Proposition}

\newtheorem{definition}[thm]{Definition}

\theoremstyle{definition}
\newtheorem{remark}[thm]{Remark}

  \newtheorem{quest}[thm]{Questions}

\def\min{\operatorname{min}}
\def\im{\operatorname{Im}}

\def\max{\operatorname{max}}

\def\c1{\operatorname{c_1}}
\def\c2{\operatorname{c_2}}

\def\Y{{\mathcal Y}}

\def\Sym{\operatorname{Sym}}

\def\PGL{\operatorname{PGL}}

\def\CC{{\mathbb C}}

\def\PP{{\mathbb P}}

\def\B{{\mathcal B}}

\def\R{{\mathcal R}}
\def\L{{\mathcal L}}
\def\M{{\rm M}}
\def\N{{\mathcal N}}
\def\O{{\mathcal O}}
\def\I{{\mathcal J}}

\def\H{{\mathcal H}}

\def\K{{\mathcal K}}
\def\U{{\mathcal U}}
\def\V{{\mathcal V}}
\def\W{{\mathcal W}}
 
\def\K{{\mathcal K}}

\def\FF{{\mathbb F}}
\def\x{\times}                   
\def\cong{\simeq}

\def\sub{\subseteq}

\def\+{\oplus}                   
\def\*{\otimes}                  

\def\Pic{\operatorname{Pic}}

\begin{document}

\title{Moduli of nodal curves on $K3$ surfaces}

\author{C. Ciliberto}
\address{Ciro Ciliberto, Dipartimento di Matematica, Universit\`a di Roma Tor Vergata, Via della Ricerca Scientifica, 00173 Roma, Italy}
\email{cilibert@mat.uniroma2.it}

\author{F. Flamini}
\address{Flaminio Flamini, Dipartimento di Matematica, Universit\`a di Roma Tor Vergata, Via della Ricerca Scientifica, 00173 Roma, Italy}
\email{flamini@mat.uniroma2.it}

\author{C. Galati}
\address{Concettina Galati, Dipartimento di Matematica e Informatica, Universit\`a della Calabria, via P. Bucci, cubo 31B, 87036 Arcavacata di Rende (CS), Italy}
\email{galati@mat.unical.it}

\author{A. L. Knutsen}
\address{Andreas Leopold Knutsen, Department of Mathematics, University of Bergen, Postboks 7800,
5020 Bergen, Norway}
\email{andreas.knutsen@math.uib.no}

\begin{abstract}    We consider modular properties of nodal curves on general $K3$ surfaces.  Let $\K_p$  be the moduli space of primitively polarized $K3$ surfaces $(S,L)$ of genus $p\geqslant 3$ and $\V_{p,m,\delta}\to \K_p$ be the universal Severi variety of $\delta$--nodal irreducible curves in $|mL|$ on  $(S,L)\in \K_p$. We find conditions on $p, m,\delta$ for the existence of an irreducible component $\V$ of $\V_{p,m,\delta}$ on which the moduli map $\psi: \V\to  \M_g$ (with $g=
m^2 (p -1) + 1-\delta$) has generically maximal rank  differential. Our results, which  for any $p$ leave only finitely many cases unsolved and are optimal for $m\geqslant 5$ (except for very low values of $p$), are summarized in Theorem \ref {thm:main} in the introduction. 
\end{abstract}

\maketitle


\section{Introduction}

Let $(S,L)$  be a smooth, primitively  polarized complex $K3$ surface of genus $p\geqslant 2$, with $L$ a globally generated, indivisible line bundle with $L^ 2=2p-2$. We denote by $\K_p$ the moduli space (or stack) of smooth primitively polarized $K3$ surfaces of genus $p$, which is  smooth, irreducible, of dimension $19$. Its elements  correspond to the isomorphism classes $[S,L]$ of pairs $(S,L)$  as above. We will often abuse notation and denote $[S,L]$ simply by $S$.

For $m \geqslant 1$,   the arithmetic genus of the curves in $|mL|$ is $p(m)=m^ 2(p-1)+1$. Let  $\delta$ be an integer such that $ 0 \leqslant \delta \leqslant p(m)$.    We consider 
the quasi--projective scheme (or stack) $\V_{p, m,\delta}$ (or simply $\V_{m,\delta}$, when $p$ is understood), called the $(m,\delta)$--\emph{universal Severi variety}, parametrizing all pairs  $(S,C)$, with  $(S,L)  \in \K_p$ and $C\in |mL|$ a reduced and  irreducible curve, with only $\delta$ nodes as singularities. 

One has the projection 
\begin{equation*}\label{eq:phiVK}
\phi_{m,\delta}: \V_{m,\delta}\to \K_{p} 
\end{equation*}
whose fiber over $(S,L) \in \K_p$ is the variety, denoted by
$V_{m, \delta} (S)$,
called the {\em Severi variety} of  $\delta$-nodal  irreducible curves  in $|mL|$.  The variety
$V_{m,\delta}(S)$ is  well-known to be  smooth, pure of dimension $g_{m,\delta}:=p(m)-\delta=m^2(p-1)+1-\delta$. 
We will often write $g$ for $g_{m,\delta}$ if no confusion arises: this is the {\em geometric genus} of  any  curve in $\V_{m,\delta}$.  

One has the obvious {\em moduli map} 
\begin{equation*}\label{eq:modmap0}
\xymatrix{\psi_{m,\delta} :\V_{m,\delta} \ar[r] &  \M_g,}  
\end{equation*}where $\M_g$ is moduli space of  smooth  genus--$g$  curves, sending a curve $C$ to the class of its normalization. Our objective in this paper is to find conditions on $p$, $m$ and $\delta$ (or equivalently $g$) ensuring the existence of a component $\V$ of $\V_{m, \delta}$, such that $\psi_{m,\delta}|_{\V}$ is either \begin{inparaenum} [(a)] \item {\em generically finite} onto its image, or \item {\em dominant} onto ${{\rm M}}_g$. \end{inparaenum} Note that (a) can happen only for $g\geqslant 11$ and (b) only for $g\leqslant 11$.

We collect our results in the following statement, which, for any $p$, solves the problem for all but finitely many $(g,m)$; for instance, for $m \geqslant 5$ and $p \geqslant 7$ the result yields that the moduli map has maximal rank on some component for any $g$.

\begin{thm}\label{thm:main} With the above notation, one has: \medskip

\noindent $(A)$  For the following values of $p \geqslant 3$, $m$ and $g$ there is an irreducible component $\V$ of $\V_{m,\delta}$, such that the moduli map $\V\to  \M_g$ is dominant:
 
 \begin{itemize}
 \item  $m=1$ and $0 \leqslant g \leqslant 7$;
  \item $m=2$,  $p \geqslant g-1$   and $0 \leqslant g \leqslant 8$;
  \item $m=3$,  $p \geqslant g-2$  and $0 \leqslant g \leqslant 9$;
   \item $m=4$,  $p \geqslant g-3$  and $0 \leqslant g \leqslant 10$;
   \item $m \geqslant 5$,  $p \geqslant g-4$  and $0 \leqslant g \leqslant 11$.
   \end{itemize}
  
\noindent $(B)$ For the following values of $p$, $m$ and $g$ there is an irreducible component $\V$ of $\V_{m,\delta}$, such that the moduli map $\V\to  M_g$ is generically finite onto its image:
 \begin{itemize}
 \item $m=1$ and $p\geqslant g \geqslant 15$;
  \item $2\leqslant m\leqslant 4$,  $p \geqslant 15$  and $g\geqslant 16$;
  \item $m\geqslant 5$,  $p \geqslant 7$   and $ g \geqslant 11$.
  \end{itemize}
\end{thm}

To prove this, it suffices  to exhibit  some specific curve in the universal Severi variety such that a component of the fiber of the moduli map at that curve has the \emph{right dimension}, i.e, $\min \{0, 22-2g\}$. To do this, we argue by degeneration, i.e., we consider partial compactifications $\overline \K_p$ and  $\overline \V_{m,\delta}$ of both $\K_p$ and  $\V_{m,\delta}$ and prove the above assertion for curves in the boundary.

 The partial compactification $\overline \K_p$ is obtained by
adding to $\K_p$ a divisor $\mathfrak S_p$ parametrizing pairs $(S,T)$, where
$S$ is a \emph{reducible $K3$ surface} of genus $p$ that can be realised in
$\PP^ p$ as the union of two rational normal scrolls intersecting along an elliptic
normal  curve $E$, and $T$ is the zero scheme of a section of the
first cotangent sheaf $T^1_{S}$, consisting of $16$ points on $E$. These $16$
points, plus a subtle deformation argument of nodal curves,  play a
fundamental role in our approach for $m=1$.  For $m>1$ we specialize curves in the Severi variety to
suitable unions of curves used for $m=1$ plus other types of limit curves,
namely tacnodal limit curves passing through two of the $16$ special points.

The paper is organized as follows. The short \S \ref {S:setup} is devoted to  preliminary results, and we in particular define a slightly broader notion of Severi variety including the cases of curves with more than $\delta$ nodes and of reducible curves.  
In \S \ref {S:stableK3} we introduce the partial compactification $\overline \K_p$ we use (cf.\,\cite{fri2,Ku,pp}). In \S \ref {sec:limnod1} we start the analysis of the case $m=1$, and we introduce the curves in $\overline \V_{1,\delta}$ over the reducible surfaces in $\mathfrak S_p$ that we use for proving our results. Section \ref {sec:modulimap} is devoted to the study of the fiber of the moduli map for the curves introduced in the previous \S \ref {sec:limnod1}: this is the technical core of this paper. In \S \ref {sec:m1} we prove the $m=1$ part of Theorem \ref {thm:main}. In  \S \ref {sec:limnod} we introduce the tacnodal limit curves mentioned above, which are needed to  work out the cases $m>1$ of 
Theorem \ref {thm:main} in \S\S \ref {sec:dominance}--\ref  {sec:genfin}.

The study of Severi varieties is classical and  closely related to modular properties. For the case of nodal plane curves the traditional reference is Severi's wide exposition in \cite [Anhang F] {se}, although already in Enriques--Chisini's famous book \cite [vol. III, chapt. III, \S 33] {enr} families of plane nodal curves with general moduli have been considered. The most important result on this subject is Harris' proof in \cite {Ha} of the so--called \emph{Severi conjecture}, which asserts that the Severi variety of irreducible plane curves of degree $d$ with $\delta$ nodes is irreducible. 

In recent times there has been a growing interest in Severi varieties for $K3$ surfaces and their modular properties. In their seminal works \cite {MM,mu0,mu} Mori and Mukai proved that in the case of smooth curves in the hyperplane section class, i.e., what we denoted here by $\V_{p,1,0}$, the modular map: \begin{inparaenum}[(a)] \item dominates $\M_p$, for $p\leqslant 9$ and $p=11$, whereas this is false for $p=10$, and  \item is generically finite between $\V_{p,1,0}$ and its image in $\M_p$ for $p=11$ and $p\geqslant 13$, whereas this is false for $p=12$. \end{inparaenum}
In \cite {clm} one gives a different proof of these results, proving, in addition, that $\V_{p,1,0}$ birationally maps to its image in $\M_p$ for $p=11$ and $p\geqslant 13$. The case of $\V_{p,m,0}$, with $m>1$, has been studied in \cite {clm1}. 

As for $\delta>0$, in \cite {fkps} one proves that $\V_{p,1,\delta}$ dominates $\M_g$ for $2\leqslant g<p\leqslant 11$. Quite recently, Kemeny, inspired by ideas of Mukai's,  and using  geometric constructions of appropriate curves on $K3$ surfaces with high rank Picard group,  proved  in \cite {Ke} that there is an irreducible component of  $\V_{p,m,\delta}$ for which the moduli map is generically finite onto its image for  all but finitely many values of $p$, $m$ and $\delta$.  Kemeny's results partly intersect with part $(B)$ of our Theorem \ref {thm:main};  his results are slightly stronger than ours for $m \leqslant 4$,  however our results are  stronger and in fact optimal  for $m\geqslant 5$ (if $p \geqslant 7$). Moreover part $(A)$, which is also optimal for $m\geqslant 5$ (except for some very low values of $p$),
 is completely new. 
Note that a  different proof of the case $m=1$ in Theorem \ref {thm:main} has recently been given in \cite{cfgk}. 

To finish, it is the case to mention that probably the most interesting open problems on the subject are the following:

\begin{quest} For $S\in \K_p$ general, is the Severi variety $V_{p,m,\delta}(S)$ irreducible? Is the universal Severi variety $\V_{p,m,\delta}$ irreducible? 
\end{quest}

So far it is only known that $\V_{p,1,\delta}$ is irreducible for  $3\leqslant g\leqslant  12$ and $g\neq 11$  (see \cite {cd}). 

\subsection*{Terminology and conventions}\label{ss:term}  We work over $\mathbb{C}$. For $X$ a Gorenstein variety, we denote by $\O_X$ and $\omega_X$ the  structure sheaf and the canonical line bundle, respectively, and $K_X$ will denote a canonical divisor of $X$. If $x\in X$, then $T_{X,x}$ denotes the Zariski tangent space to $X$ at $x$. For $Y \subset X$ a subscheme, $\I_{Y/X}$ (or simply  $\I_{Y}$ if there is no danger of confusion) will denote its ideal sheaf whereas $\N_{Y/ X}$ its normal sheaf. For line bundles we will sometimes abuse notation and use the additive notation to denote tensor products. Finally, we will denote by $L^*$ the inverse of a line bundle $L$ and by $\equiv$ the linear equivalence of Cartier divisors. 

\subsection*{Acknowledgements} The first three authors  have been supported by the GNSAGA of Indam and by the PRIN project ``Geometry of projective varieties'', funded by the Italian MIUR.

The authors wish to thank M. Kemeny for interesting and useful conversations on the subject of this paper.   We also thank the referee for a very careful reading and for making several observations and suggestions that improved the readability of the paper, as well as discovering some inaccuracies in the first version.

\section{Preliminaries }\label{S:setup}

 To prove our main results, we will need  to consider nodal, {\it
reducible} curves. To this end we will work with a slightly broader notion of {\it Severi variety} than the one from the introduction. 

Let $|D|$ be a base point free complete linear system on a smooth $K3$ surface $S$. Consider the quasi-projective scheme $V_{|D|,\delta}(S)$ parametrizing pairs $(C,\nu_C)$ such that\\
\begin{inparaenum}
\item $C$ is a reduced (possibly reducible) nodal curve in $|D|$;\\
\item  $\nu_C$ is a subset of $\delta$ of its nodes, henceforth called the {\it marked nodes} (called {\it assigned nodes} in \cite{Tan}), such that
 the normalization $\widetilde C\to C$ at $\nu_C$ is $2$-connected.
\end{inparaenum}

Then, by \cite[Thms. 3.8 and 3.11]{Tan}, one has:\\
\begin{inparaenum}
\item[(i)] $V_{|D|,\delta}(S)$ is smooth of codimension $\delta$ in $|D|$;\\
\item[(ii)] in any component of $V_{|D|,\delta}(S)$, the general pair $(C,\nu_C)$ is such that $C$ is irreducible with precisely $\delta$ nodes.
\end{inparaenum}

\noindent We leave it to the reader to verify that the conditions in \cite{Tan} are in fact equivalent to ours. We call $V_{|D|,\delta}(S)$ the {\it Severi variety} of nodal curves in $|D|$ with $\delta$ marked nodes. This definition is different from others one finds in the literature (and even from the one in the introduction!). Usually, in the Severi variety  one considers only irreducible curves $C$ with exactly $\delta$ nodes. With our definition we consider the desingularization of a partial compactification of the latter variety.
This will be useful for our purposes.

We define the $(m,\delta)$-universal Severi variety $\V_{p,m,\delta}$ (or simply $\V_{m,\delta}$ when $p$ is understood) to be the quasi-projective variety (or stack) parametrizing  triples $(S,C,\nu_C)$, with $S=(S,L) \in \K_p$ and $(C,\nu_C) \in V_{|mL|,\delta}(S)$. It is smooth and pure of dimension $19+p(m)-\delta=19+g_{m,\delta}$ and the general element in any component is a triple $(S,C,\nu_C)$ with $C$ irreducible with exactly 
$\delta$ nodes. We may simplify notation and identify 
$(S,C,\nu_C)\in \V_{p, m,\delta}$ with the curve $C$, when the surface $S$ and the set of nodes $\nu_C$ are intended.   

One has the projection 
\[ 
\xymatrix{\phi_{m,\delta}: \V_{m,\delta} \ar[r] & \K_p} \] 
whose fiber over $S \in \K_p$ is the Severi variety $V_{m,\delta}(S):=V_{|mL|,\delta}(S)$. Similarly, we have a {\em moduli map} 
\begin{equation*}\label{eq:modmap}
\xymatrix{\psi_{m,\delta} :\V_{m,\delta} \ar[r] & \overline{{\rm M}}_{g}} 
\end{equation*}
(where $\overline{{\rm M}}_{g}$ is the moduli space of genus $g$ stable curves
and we recall that $g=g_{m,\delta}:=p(m)-\delta$),
which sends a curve $C$ to the stable model of the partial normalization of $C$ at the $\delta$ marked nodes.

For $S\in \K_p$ and integers $0 \leqslant \delta\leqslant \delta'  \leqslant p(m)$, one has a correspondence
\[
X_{m, \delta, \delta'}(S):=\{(C,C')\in V_{m,\delta}(S)\times V_{m,\delta'}(S)\, | \, C=C' \,\,\, \text{and}\,\,\, \nu_C\subseteq \nu_{C'}\}
\]
with the two projections 
\[
\xymatrix{
V_{m,\delta}(S)&& X_{m, \delta, \delta'}(S) \ar[rr]^{p_2} \ar[ll]_{p_1}&& V_{m,\delta'}(S), 
 }
\]
which are both finite onto their images. Precisely:\\
\begin{inparaenum}[$\bullet$]
\item $p_2$ is surjective, \`etale of degree $\delta' \choose{\delta}$, hence $\dim(X_{m, \delta, \delta'}(S))=g_{m,\delta'}$;\\
\item $p_1$ is birational onto its image, denoted by  $V_{[m,\delta,\delta']}(S)$, which is pure with
 \[
 \dim (V_{[m,\delta,\delta']}(S))=g_{m,\delta'}=\dim(V_{m,\delta}(S))-(\delta'-\delta).\]
 \end{inparaenum}
 Roughly speaking, the variety $V_{[m,\delta,\delta']}(S)$ is the proper subvariety of $V_{m,\delta}(S)$ consisting of curves with at least $\delta'$ nodes, $\delta$ of them marked. The general point of any component of $V_{[m,\delta,\delta']}(S)$ corresponds to a curve with exactly $\delta'$ nodes. So one has the filtration
\[
V_{[m,\delta,p(m)]}(S)\subset V_{[m,\delta,p(m)-1]}(S)\subset \ldots \subset V_{[m,\delta,\delta+1]}(S)\subset V_{[m,\delta,\delta]}(S) =V_{m,\delta}(S)
\]
in which each variety has codimension $1$ in the subsequent. 

\begin{remark}\label{rem:filtr} Given a component $V$ of $V_{m,\delta}(S)$ and 
$\delta'>\delta$, there is no   a priori 	guarantee that $V\cap V_{[m,\delta, \delta']}(S)\neq \emptyset$. If this is the case, then each component of $V\cap V_{[m,\delta, \delta']}(S)$ has codimension $\delta'-\delta$ in $V$ and 
we say that $V$ is \emph{$\delta'$--complete}. If $V$ is $\delta'$--complete, then it is also $\delta''$--complete for $\delta<\delta''<\delta'$. If $V$ is $p(m)$--complete we say it is \emph{fully complete}, i.e., $V$ is fully complete if and only if it contains a point parametrizing a rational nodal curve. 

If $V$ is a component of $V_{m,\delta}(S)$ and $W$ a component of $V_{m,\delta'}(S)$, such that $\dim(X_{m,\delta, \delta'}\cap (V\times W))=g_{m,\delta'}$, then $X_{m,\delta, \delta'}\cap (V\times W)$ dominates a component $V'$ of $V_{[m,\delta,\delta']}$ contained in $V$, whose general point is a curve in $W$ with $\delta$ marked nodes. In this case we will abuse language and say that $W$ is \emph{included} in $V$. 
\end{remark}

Of course one can make a relative version of the previous definitions, and make sense of 
the subscheme $\V_{[m,\delta,\delta']}\subset \V_{m,\delta}$, which has dimension
$19+g_{m,\delta'}$, of the filtration 
\[
\V_{[m,\delta,p(m)]}\subset \V_{[m,\delta,p(m)-1]}\subset \ldots \subset \V_{[m,\delta,\delta+1]}\subset \V_{m,\delta},
\]
of the definition of a $\delta'$--complete component $\V$ of $\V_{m,\delta}$, for $\delta'>\delta$, of fully complete components,  etc.

For $\delta'>\delta$,  the image of $\V_{[m,\delta,\delta']}$ via $\psi_{m,\delta}$ sits in the $(\delta'-\delta)$--codimensional locus $\Delta_{g_{m,\delta},\delta'-\delta}$ of  $(\delta'-\delta)$-- nodal curves in $\overline \M_{g_{m,\delta}}$.

\section{Stable limits of $K3$ surfaces}\label{S:stableK3} 

In this section we consider some reducible surfaces (see \cite{clm}), which are {\em limits} of smooth, polarized $K3$ surfaces, in the sense of the following:

\begin{definition}\label{deformation}
Let $R$ be a compact, connected analytic variety. A variety 
$\Y$ is said to be a {\rm deformation of $R$} if there exists a proper, flat morphism $$\pi:\Y\to\mathbb D=\{t\in\mathbb C | \mid t\mid<1\}$$ such that 
$R=\Y_0:=\pi^*(0)$. Accordingly, $R$ is said to be a {\rm flat limit of $\Y_t:= \pi^*(t)$}, for $t \neq 0$. 

If $\Y$ is smooth and if any of the $\pi$-fibers has at most normal crossing singularities, $\Y\stackrel{\pi}{\longrightarrow} \mathbb D$ is 
said to be a {\rm{semi-stable deformation of $R$}} (and $R$ is a {\rm semi-stable limit} of $\Y_t$, $t \neq 0$). If, in addition, $\Y_t$ is a smooth $K3$ surface, for $t \neq 0$, 
then $R$ is a {\rm semi-stable limit of $K3$ surfaces}.

If, in the above setting, one has a line bundle $\L$ on $\Y$, with $L_t=\L_{|\mathcal Y_t}$ for $t \neq 0$ and $L=\L_{|R}$, then one says that $(R,L)$ is a \emph{limit} of $(\mathcal Y_t,L_t)$ for $t \neq 0$. 
\end{definition}

Let $p=2l+\varepsilon\geqslant 3$ be an integer with $\varepsilon=0,1$ and $l \in \mathbb N$. If $E' \subset \PP^{p}$ is an elliptic normal curve of degree $p+1$, we set $L_{E'}:=\O_{E'}(1)$. 
Consider two {\it general} line bundles $L_1, L_2 \in \Pic^2(E') $ with $L_1\neq L_2$.
We denote by $R'_i$ the rational normal scroll of degree $p-1$ in
$\PP^p$ described by the secant lines to $E'$ spanned by the divisors in $|L_i|$, for $1\leqslant i\leqslant 2$. We have 
\[ R'_i \cong  
\begin{cases} 
\PP^1 \times \PP^1 &  \; \mbox{if $p=2l+1$ is odd and $\O_{E'}(1) \not \sim (l+1)L_i$, for $i=1,2$ } \\ 
 \mathbb{F}_1 &  \; \mbox{if $p=2l$ is even}.
\end{cases} 
\] 

The surfaces $R'_1$ and $R'_2$ are $\PP^ 1$--bundles on $\PP^ 1$. 
  We denote by $\sigma_i$ and $F_i$ a minimal section and a fiber of the ruling of $R_i'$, respectively, so that $\sigma_i^ 2=\varepsilon -1$ and $F_i^ 2=0$, and   \begin{equation}
  \label{eq:A}
L_{R'_i}:=\O_{R'_i}(1) \simeq \O_{R'_i}(\sigma_i + lF_i), \;\;\; \text{for}\;\;\;  1 \leqslant i \leqslant 2.
\end{equation}

By \cite[Thm.\;1]{clm}, $R'_1$ and $R'_2$  intersect transversely along $E'$, which is anticanonical on $R'_i$, i.e.
 \begin{equation}
  \label{eq:B1}
E'\equiv -K_{R'_i} \equiv 2\sigma_i+(3-\epsilon)F_i \;\;\; \text{for}\;\;\;  1 \leqslant i \leqslant 2.
\end{equation}
Hence $R'=R'_1\cup R'_2$ has normal crossings and $\omega_{R'}$ is trivial.  We set $L_{R'}:=\O_{R'}(1)$. 
The first cotangent sheaf $T^1_{R'}$ (cf. \cite[\S\;1]{fri})  is the degree $16$ line bundle on $E'$
 \begin{equation}
  \label{eq:B}
T^1_{R'}\cong \N_{E'/R'_1}\otimes \N_{E'/R'_2}\cong L_{E'}^ {\otimes 4}\otimes (L_1\otimes L_2)^ {\otimes (3-2l-\epsilon)},
\end{equation}
the last isomorphism coming from \eqref {eq:A} and \eqref {eq:B1}.

The surface $R'$ is a flat limit of smooth $K3$ surfaces in $\PP^p$. Namely, if $\H_p$ is the component of the Hilbert scheme of surfaces in $\PP^ p$ containing $K3$ surfaces $S$ such that $[S, \O_S(1)]\in \K_p$, then $R'$ sits in $\H_p$ and, for general choices of $E', L_1, L_2$, the Hilbert scheme $\H_p$ is smooth at $R'$ (see  \cite{clm}).  However, the fact that $T^1_{R'}$ is non-trivial implies that $R'$ is not a semi--stable limit of $K3$ surfaces: indeed, the total space of every flat deformation of $R'$ to $K3$ surfaces in $\H_p$ is singular along a divisor $T\in |T^1_{R'}|$ (cf. \cite[Prop. 1.11 and \S\;2]{fri}).  More precisely (see again \cite{clm} for details), if  
\begin{equation}\label{eq:defo}
\xymatrix{
R'\ar[d]\,\, \ar@{^{(}->}[r]& \R'\ar[d]^ {\pi'} \ar@{^{(}->}[r] \ar[dr]^ {{\rm pr_2}} &\mathbb D\times \PP^ p\ar[d] \\\
0\,\, \ar@{^{(}->}[r]& \mathbb D&\PP^ p
 }
\end{equation}
is a deformation of 
$R'$ in $\mathbb P^p$ whose general member is a smooth $K3$ surface, then  $\R'$ has  double points at the points of a divisor $T\in \vert T^1_{R'}\vert$ 
associated to tangent direction to $\H_p$ at $R'$ determined by the  deformation \eqref {eq:defo}, via the standard map
\begin{equation}\label{standard}
T_{\H_p, R'}\cong H^ 0(R', \N_{R'/\PP^ p})\to H^ 0(T^ 1_{R'}).
\end{equation}
If $T$ is reduced (this is the case if \eqref {eq:defo} is general enough), then  the tangent cone to $\R'$ at each of the 16 points of $T$ has rank 4. 
In this case, by blowing up $\R'$ at the points of $T$, the exceptional divisors are rank $4$ quadric surfaces and, 
by contracting each of them along a ruling on one of the two irreducible components of the strict transform of $R'$, one obtains
a small resolution of singularities $\Pi: \R \to \R'$ and a 
semi-stable degeneration $\pi: \R {\longrightarrow}\mathbb D$ of $K3$ surfaces, with central fiber $R := R_1\cup R_2$, where  $R_i=\Pi^{-1}(R'_i)$, for $i=1,2$ and still $\omega_R\cong \O_R$.  Note that we have a line bundle
on $\R$, defined as  $\L:=\Pi^ *({\rm pr}_2^ *(\O_{\PP^ p}(1)))$. So $\pi: \R {\longrightarrow}\mathbb D$ is a deformation of polarized $K3$ surfaces, and we set 
$L_R:=\L_{|R}=\Pi_{|R}^ *(L_{R'})$.

We will abuse notation and terminology by identifying curves on $R'$ with their proper transforms on $R$ (so we may talk of \emph{lines} on $R$, etc.). 

We will set  $E = R_1 \cap R_2$; then $E\cong E'$ (and we will often identify them). We have $T^1_{R}\simeq \O_E$. 
If the divisor $T$ corresponding to the deformation  \eqref {eq:defo} is not reduced, the situation can be handled in a similar way, but we will not dwell on it, because we will not need it. 

The above limits $R$ of $K3$ surfaces are  \emph{stable},  Type II degenerations according to the Kulikov-Persson-Pinkham classification  of semi-stable degenerations of $K3$ surfaces (see \cite{Ku, pp}).

 By \cite[Thm. 4.10]{fri2} there is a normal, separated partial compactification $\overline{\K}_p$ of $\K_p$ obtained by adding to $\K_p$ a smooth divisor consisting of various components corresponding to the various kinds of 
Type II degenerations of  $K3$ surfaces. One of these components, which we henceforward call $\mathfrak{S}_p$,
corresponds to the degenerations we mentioned above. Specifically, points of $\mathfrak{S}_p$ parametrize  isomorphism classes of  pairs $(R',T)$ with $R'=R'_1 \cup R'_2$ as above and $T \in |T^1_{R'}|$, cf. \cite[Def. 4.9]{fri2}.  Since all our considerations will be local around general members of $\mathfrak{S}_p$, where the associated $T$ is reduced, we may and will henceforth assume  (after substituting $\overline{\K}_p$ and $\mathfrak{S}_p$ with dense open subvarieties) that $\overline{\K}_p$ is smooth, 
\[ \overline{\K}_p = \K_p \cup \mathfrak{S}_p \]
and that $T$ is reduced for all $(R',T) \in \mathfrak{S}_p$. Thus, again since  all our  considerations will be local around such pairs, we may identify $(R',T)$ with the surface $R=R_1 \cup R_2$, as above, with
$R_1\cong R'_1$ and $R_2$  the blow--up of $R'_2$ at the $16$ points of $T$ on $E'$. If $T=p_1+\ldots+p_{16}$, we denote by $\mathfrak{e}_i$ the exceptional divisor on $R_2$ over $p_i$, for $1\leqslant i\leqslant 16$, and 
we set $\mathfrak{e}:=  \mathfrak e_1+\cdots+\mathfrak e_{16}$.

To be explicit, let us denote by $\R'_p\subset \H_p$ the locally closed subscheme whose points correspond to unions of scrolls $R'=R'_1\cup R'_2$ with $L_1,L_2$ general as above  and such that $(R',T) \in \mathfrak{S}_p$ for some
$T \in |T^1_{R'}|$. 
 One has an obvious dominant  morphism
$\R'_p \to {\rm M}_1 $
whose fiber over the class of the elliptic curve $E'$, modulo projectivities ${\rm PGL}(p+1,\mathbb C)$, is an open subset of ${\rm Sym}^ 2({\rm Pic}^ 2(E'))\times {\rm Pic}^ {p+1}(E')$ modulo the action of ${\rm Aut}(E')$. Hence 
\[
\dim (\R'_p)=p^ 2+2p+3,\,\,\, \text{whereas}\,\,\, \dim(\H_p)=p^ 2+2p+19, 
\]
so that for $R'\in \R'_p$ general, the \emph{normal space} $N_{R'/ \R'_p}=T_{\H_p,R'}/ T_{\R'_p, R'}$ has dimension 16.
The map \eqref{standard} factors through a map
\[
N_{R'/\R'_p} \to H^ 0(T^ 1_{R'})
\]
which is an isomorphism (see \cite {clm}).

We denote by $\R_p$ the ${\rm PGL}(p+1,\mathbb C)$--quotient of $\R'_p$, which, by the above argument, has dimension $\dim(\R_p)=3$.
 By definition, there is a surjective   morphism
\[
\pi_p: \mathfrak{S}_p\to \R_p
\]
whose fiber over (the class of) $R'$ is a dense, open subset of  $|T^ 1_{R'}|$, which has dimension $15$ (by \eqref {eq:B}, given $L_1,L_2$ and $L_{E'}$, then $\O_{E'}(T)$ is determined).  This, by the way, confirms that $\dim( \mathfrak{S}_p)=18$. 

 The universal Severi variety $\V_{m,\delta}$ has  a partial compactification $\overline{\V}_{m,\delta}$ (see \cite[Lemma 1.4]{ck}), with a morphism 
\[ \overline \phi_{m,\delta}: \overline {\V}_{m,\delta} \to \overline{\K}_p\] 
extending $\phi_{m,\delta}$, where the fiber $\overline{V}_{m,\delta}(R)$
over a general 
$R =R_1\cup R_2 \in \mathfrak{S}_p$ consists of  all nodal curves $C\in |mL_R|$, with $\delta$ marked, non--disconnecting nodes in the smooth locus of $R$, i.e., the partial normalization of $C$ at the $\delta$ nodes is connected. 
(There exist more refined partial compactifications of $\V_{m,\delta}$, for instance by adding curves with tacnodes along $E$, as considered in \cite {C, galknu}, following  \cite{ran}; we will consider such curves in \S \ref {sec:limnod}, but we do not need them now). 
The total transform $C$ of a nodal curve $C'\in |mL_{R'}|$ with $h$ marked nodes in the smooth locus of $R'$ and $k$ marked nodes at points of $T$, with $h+k=\delta$,  lies in $\overline{V}_{m,\delta}(R)$, since it contains
the exceptional divisors on $R_2$ over the $k$ points of $T$, and has a marked  node on each of them off $E$.  As for the smooth case, $\overline{V}_{m,\delta}(R)$, is smooth, of dimension $g=p(m)-\delta$. We can also consider $\delta'-\delta$ codimensional subvarieties of $\overline{V}_{m,\delta}(R)$ of the form $\overline{V}_{[m,\delta,\delta']}(R)$, with $\delta<\delta'$, etc. 

Finally, there is an \emph{extended moduli map}
\[ \overline \psi_{m,\delta}: \overline \V_{m,\delta} \longrightarrow \overline{\rm M}_g.\]

We end this section with a definition, related to the construction of the surfaces $R'$, that we will need later. 

\begin{definition} \label{def:nonpiaceraaconcy}
  Let $E$ be a smooth elliptic curve with two degree-two line bundles $L_1$ and $L_2$ on it. 

For any integer $k \geqslant 0$, we define the automorphism $\phi_{k,E}$ on $E$ that sends $x \in E$ to the unique point $y \in E$ satisfying
\begin{eqnarray}
  \label{eq:relgen-odd}
  \O_E(x+y) & \cong &L_2^ { \otimes \frac{k+1}{2}}\otimes (L_1^ *)^{\otimes \frac{k-1}{2}}, \; \; \mbox{$k$ odd;} \\
\label{eq:relgen-even}  \O_E(x-y) &\cong & (L_2\otimes L_1^ *)^{\otimes  \frac{k}{2}} , \; \; \mbox{$k$ even.} 
\end{eqnarray}

For any $x \in E$, we define the effective divisor $D_{k,E}(x)$ to be the degree $k+1$ divisor
\begin{equation}
  \label{eq:defD}
  D_{k,E}(x)=x+\phi_{1,E}(x) + \cdots + \phi_{k,E}(x).
\end{equation}
\end{definition}

\section{Limits of nodal hyperplane sections on reducible $K3$ surfaces} \label{sec:limnod1}

Given $R\in \mathfrak{S}_p$  , with  $R=R_1\cup R_2$  as explained in the previous section,  we will now describe certain curves in $\overline{\V}_{1,\delta}$  lying on $R$.

\begin{definition} \label{def:compopartenza} Let 
$d, \ell$ be non--negative integers such that $\ell \leqslant 16$. Set $\delta=d+\ell$ and assume $\delta \leqslant p-\epsilon$, where 
\[ \epsilon= 
\begin{cases} 
1 &  \; \mbox{if}\,\,\,  \ell=0, \\ 
0 & \; \mbox{if}\,\,\,  \ell >0.\\ 
\end{cases} 
\] 
We define  $W_{d,\ell}(R)$ to be the set of  reduced  curves in   $\overline{V}_{1,\delta}(R)$:\\
\begin{inparaenum} [(a)]
\item having exactly $d_1:=\lfloor \frac{d}{2} \rfloor$ nodes on $R_1-E$ and $d_2:=\lceil \frac{d}{2} \rceil$ nodes on 
$R_2-E-\mathfrak e$, hence they split off $d_i$ lines on $R_i$, for $i=1,2$;\\
\item such that the union of these $d=d_1+d_2$ lines is connected;\\
\item   containing exactly $\ell$ irreducible components of $\mathfrak e$, so they have  $\ell$ further nodes, none of them lying on $E$. 
\end{inparaenum}

These curves have exactly $\delta$ nodes on the smooth locus of $R$, which are the {\em marked} nodes. We write $W_{\delta}(R)$ for $W_{\delta,0}(R)$. 

For any curve $C$ in $W_{d,\ell}(R)$, we denote by $\mathfrak C$ the connected union of $d$ lines as in (b), called the \emph{line chain} of $C$, and by $\gamma_i$
the irreducible component of the residual curve to $\mathfrak C$ and to the $\ell$ components of the exceptional curve $\mathfrak e$ on $R_i$, for $i=1,2$.\end{definition}

Curves in  $W_{d,\ell}(R)$ are total transforms of curves on $R'=\pi_p(R)$  with $d$ marked nodes on $R'\setminus E'$ and passing through $\ell$ marked points in $T$.   
Members of $W_{\delta}(R)$ are shown in the following pictures, which also show their images via the moduli map $\overline \psi_{1,\delta}$ (provided $g \geqslant 3$ if $\delta$ is odd and $g \geqslant 2$ if $\delta$ is even, otherwise the image in moduli is different). The curves $\gamma_i \subset R_i$, $i=1,2$,  are mapped each to one of the two rational components of the image curve in $\overline{\M}_{g}$. The situation is similar for curves in $W_{d,\ell}(R)$, with $\ell>0$. 
\begin{figure}[ht] 
\[\begin{array}{ll}
\includegraphics[width=8cm]{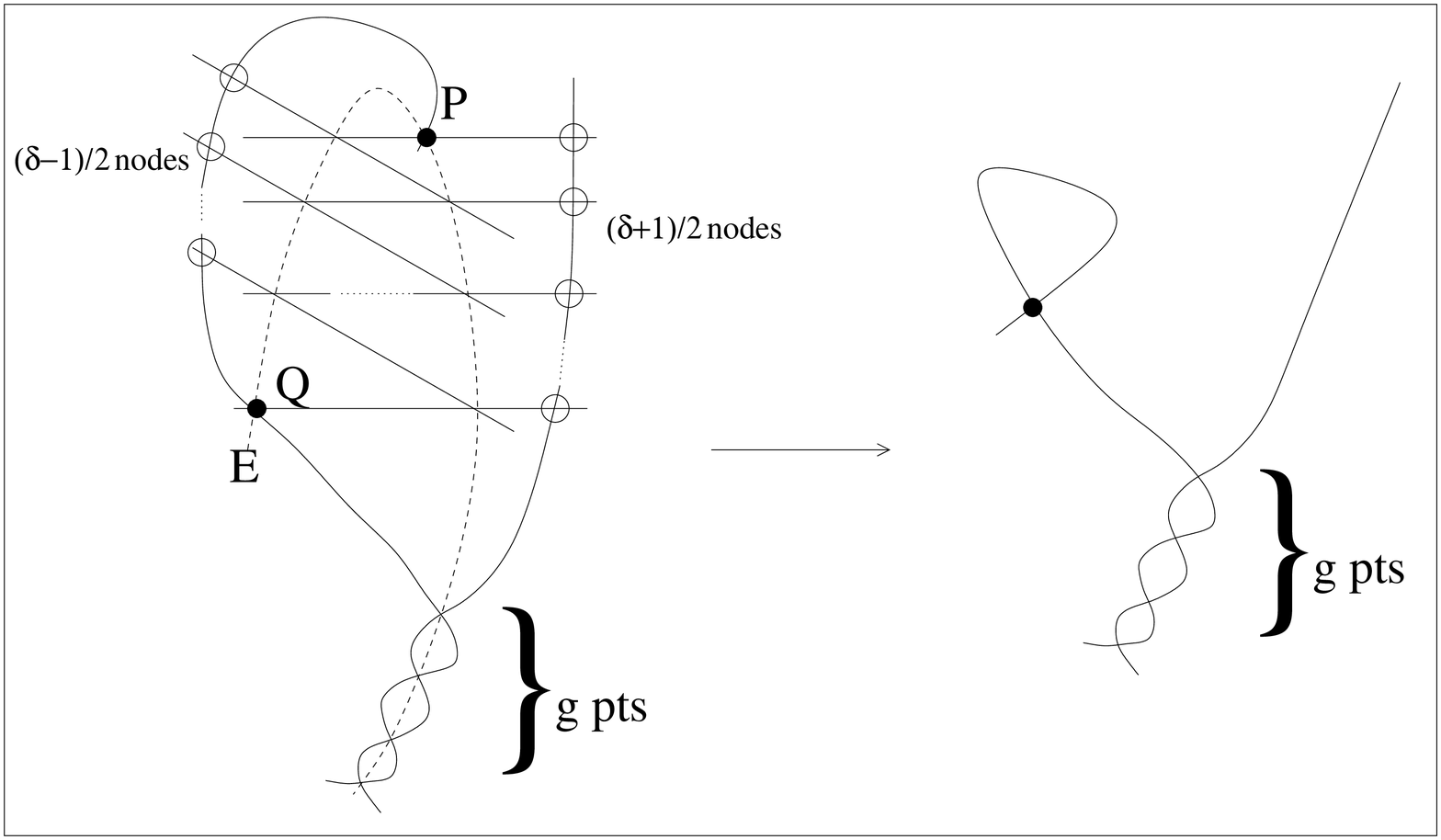} & \includegraphics[width=8cm]{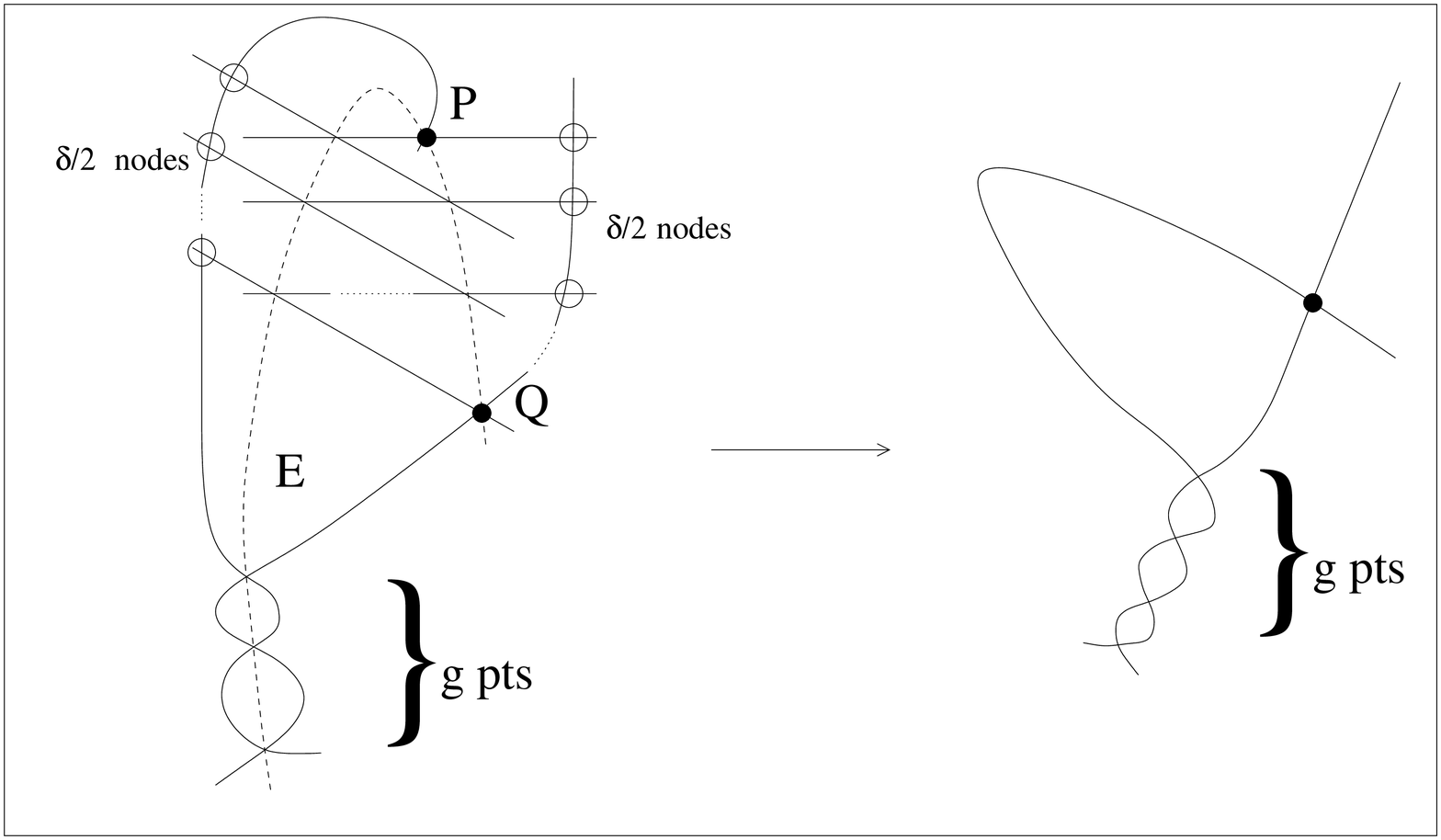} \\
\end{array}
\]
\caption{Members of $W_{\delta}(R)$ when $\delta$ is odd (left) and 
even (right).}
\label{fig:dis}
\end{figure}

The condition  $\delta \leqslant p-\epsilon$ guarantees that the normalization of   a curve in $W_{d, \ell}(R)$  at the $\delta$ marked nodes is connected.
  
The points $P$ and $Q$ in the picture are called {\it the distinguished points of the curve in $W_{d,\ell}(R)$}. They
satisfy the relation $\phi_{d,E}(P)=Q$ (equivalently, $\phi_{d,E}(Q)=P$ if $d$ is odd).
Note that $P,Q \in \gamma_1$ if $\delta$ is odd, whereas $P \in \gamma_1$ and $Q \in \gamma_2$ if $\delta$ is even.

We will consider the pair of distinguished points as an element $P+Q \in \Sym^2(E)$ when  $d$ is odd, whereas we will consider the pair as an ordered element $(P,Q) \in E \x E$ when  $d$ is even, since we then can distinguish $P$ and $Q$, as $P\in \gamma_1$ and $Q\in \gamma_2$.

Recalling \eqref{eq:defD}, we note that $D_{d,E}(P)$ (equivalently, $D_{d,E}(Q)$ if $d$ is odd) is the (reduced) divisor of intersection points of the line chain of the curve in $W_{\delta}(R)$ with $E$.

\begin{proposition}\label{prop:rightdim} Assume  $0 \leqslant \delta \leqslant p-1$.
 Then $W_{\delta}(R)$ is a smooth, dense open subset of a component of dimension $g=p-\delta$ of $\overline V_{1,\delta}(R)$.
\end{proposition}

\begin{proof} If $\delta$ is even we denote by
$G_{\frac{\delta}{2}} \subset E \x E$  the graph of the translation by $\frac{\delta}{2}(L_1-L_2)\in {\rm Pic}^ 0(E)$. Then we have surjective morphisms
\[
g_{p,\delta}: W_{\delta}(R) \longrightarrow
\begin{cases}
P+Q\in |\frac{\delta+1}{2}L_2-\frac{\delta-1}{2}L_1|  & \mbox{if $\delta$ is odd}, \\
(P,Q)\in G_{\frac{\delta}{2}} & \mbox{if $\delta$ is even.} 
\end{cases}
\]
Take any point $\eta=P+Q$ or $\eta=(P,Q)$ in the target of $g_{p,\delta}$. Set $D_{\eta}:=D_{\delta,E}(P)$, which is an effective divisor of degree $\delta+1$ on $E$. Then $D_{\eta}$ defines a unique line chain $\mathfrak C_{\eta}$ on $R$ and  
the fiber $g_{p,\delta}^{-1}(\eta)$ is a dense, open subset of the projective space $|L\otimes \I_{\mathfrak C_{\eta}/R}| \cong |L'\otimes \I_{\mathfrak C_{\eta}/R'}|\cong |L_{E'}\otimes \O_{E'}(-D_{\eta})|$.
Since
\begin{equation*}
  \label{eq:dimfiber}
f:=\dim (|L_{E'}\otimes \O_{E'}(-D_{\eta})|)=p- \delta-1 \geqslant 0,
\end{equation*}
then $ W_{\delta}(R)$ is
irreducible of dimension $f+1=p-\delta=g$. The smoothness of $\overline V_{1,\delta}(R)$ at the points of $W_{\delta}(R)$ follows, as usual, because  the $\delta$
marked nodes are non--disconnecting (cf. \cite[Rmk.\,1.1]{ck}).  
\end{proof}

The same argument proves the following:

\begin{proposition}\label{prop:rightdim2} Let $d,\, \ell$ be non--negative integers such that $0<\ell \leqslant 16$. Set $\delta=d+\ell$ and assume $\delta \leqslant p$.  Then $W_{d,\ell}(R)$ is a union of smooth, open dense subsets of components of dimension $g=p-\delta$ of $\overline V_{1,\delta}(R)$. 
\end{proposition}

The number of components of $W_{d,\ell}(R)$ is ${16} \choose l$, depending on the choice of subdivisors of degree $l$ of $T$.  

We can finally consider the universal family $\W_{d,\ell}$, irreducible by monodromy,  parametrizing all pairs $(R, C)$ with $R\in \mathfrak{S}_{p}$ and $C\in W_{d,\ell}(R)$  with the map $\W_{d,\ell}\to \mathfrak{S}_{p}$, whose general fiber is smooth of dimension $g$, so that $\dim(\W_{d,\ell})=g+18$.   
Note that  $\overline \V_{1,\delta}$ is smooth along $\W_{d,\ell}$. 

We will  only be concerned with the cases $\ell=0,1,2$, and mostly with $\ell=0$.

\begin{definition} \label{def:componente}
 For $\delta \leqslant p$, we denote by $\V^*_{\delta}$ the unique irreducible component of $\overline \V_{1,\delta}$ containing $\W_{\delta}$ if $\delta<p$ and $\W_{p-1,1}$ if $\delta=p$. We denote by 
  $V^ *_{\delta}(S)$ the fiber of  $\V^*_{\delta}$ over $S \in \K_p$. We denote by $\psi^*_{\delta}: \V^*_{\delta} \to \overline{\M}_g$ and  $\phi^*_{\delta}: \V^*_{\delta} \to \overline{\K}_p$ the restrictions of $\overline{\psi}_{1,\delta}$ and $\overline{\phi}_{1,\delta}$, respectively. 
\end{definition}

\begin{remark}\label{rem:imp} It is useful to notice  that, for  $R \in \mathfrak S_p$  one has
\[
W_{\delta+1}(R)\subset \overline{W_\delta(R)}, \,\,\, \text {for} \,\,\, 0\leqslant \delta\leqslant p-1,
\]and 
\[
W_{\delta,1}(R)\subset \overline{W_{\delta,0}(R)}=\overline{W_\delta(R)}, \,\,\, \text {for} \,\,\, 0\leqslant \delta\leqslant p-1.
\]
Hence we also have 
\[
\V^*_{\delta+1}\subset \V^*_{\delta}, \,\,\, \text {for} \,\,\, 0\leqslant \delta\leqslant p-1,
\]
i.e., $\V^ *_\delta$ is fully complete (see Remark \ref {rem:filtr}).  It follows by monodromy that for general $S \in \K_p$, every component of $V^ *_\delta(S)$ is fully complete. 

\end{remark}

\section{Special fibers of  the moduli map } \label{sec:modulimap} 

This section is devoted to the study of the dimension of the  fibers of the moduli map $\psi^*_{\delta}$ over curves arising from reducible $K3$ surfaces. 

For any $R' \in \R_p$, we define varieties of nodal curves analogously to Definition \ref{def:compopartenza}: we denote by $W_{\delta}(R')$ the set of reduced curves in $|L_{R'}|=|\O_{R'}(1)|$ having  exactly $\lfloor \frac{\delta}{2} \rfloor$ nodes on $R'_1-E'$ and $\lceil \frac{\delta}{2} \rceil$ nodes on 
$R'_2-E'$, so that they split off a total of $\delta$ lines, and such that the union of these $\delta$ lines is connected. Clearly, if $R'=\pi_p(R)$ for some $R=(R',T) \in \mathfrak{S}_p$, then there is a dominant, injective map
\[ \xymatrix{
W_{\delta}(R) \ar[r] & W_{\delta}(R')
}
\]
mapping a curve in $W_{\delta}(R)$ to its isomorphic image on $R'$. The image of this map is precisely the set of curves in $W_{\delta}(R')$ not passing through any of the points in $T$. 

The concepts of {\it line chain}, {\it distinguished pair of points} and curves $\gamma_i$ associated to an element of $W_{\delta}(R')$ are defined in the same way as for the curves in $W_{\delta}(R)$.

As above, we can consider the irreducible variety $\W'_{\delta}$ parametrizing all pairs $(R',C')$ with $R' \in \R_p$ and $C' \in W_{\delta}(R')$ and there is an obvious (now surjective!) map $\W_{\delta} \to \W'_{\delta}$. 

More generally, we may define a universal Severi variety of nodal curves with $\delta$ marked nodes on the smooth locus over $\R_p$, and we denote by $\V^*_{\delta,\R_p}$ the component containing $\W'_{\delta}$. There is a  modular map $\psi'_{\delta}: \V^*_{\delta,\R_p} \to \overline{\M}_g$ as before. Moreover, defining $\V^*_{\delta,\mathfrak{S}_p}$ to be the open set of curves in $(\phi^*_{\delta})^{-1}(\mathfrak{S}_p) \cap \V^*_{\delta}$ not containing any exceptional curves  in $\mathfrak e$ (note that $\W_{\delta} \subset \V^*_{\delta,\mathfrak{S}_p}$),  we have the surjective map
\[ \xymatrix{\pi_{\delta}: \V^*_{\delta,\mathfrak{S}_p} \ar[r] &   \V^*_{\delta,\R_p}} \]
 sending  $(C,R) \in \V^*_{\delta,\mathfrak{S}_p}$, where $C \subset R$ and $R=(R',T) \in \mathfrak{S}_p$, to $(C',R')$, where $C'$ is the  image of $C$ in $R'$ under the natural contraction map $R \to R'$. 

Summarizing,
we have the commutative diagrams

\begin{equation} \label{eq:diagrammonefichissimo}
\xymatrix{
&& & \overline{\M}_g \\
& \V^*_{\delta,\R_p} \ar[d] \ar[urr]^{\psi'_{\delta}} & \V^*_{\delta,\mathfrak{S}_p}   \ar@{->>}[l]^{\pi_{\delta}} \ar@{^{(}->}[r] 
\ar[d] \ar[ur] & \V_{\delta}^* \ar[d]^{\phi^*_{\delta}} 
\ar[u]_{\psi^*_{\delta}}
\\
&\R_p  & \mathfrak{S}_{p}  \ar@{->>}[l]^{\pi_p} \ar@{^{(}->}[r] & \overline{\K}_p               
}
\end{equation}
and
\begin{equation} \label{eq:diagrammonefichissimo2}
\xymatrix{&\W'_{\delta} \ar@{^{(}->}[d] & \W_{\delta} \ar@{^{(}->}[d]    \ar@{->>}[l]_{{\pi_\delta}_{|\W_\delta}} \\
& \V^*_{\delta,\R_p}  & \V^*_{\delta,\mathfrak{S}_p}   \ar@{->>}[l]^{\pi_{\delta}}            
}
\end{equation}

The fibers of $\pi_{\delta}$ are all $15$-dimensional (isomorphic to a dense open subset of $|T^1_{R'}|$). This shows that the fibers of $\psi^*_{\delta}$ over curves that come from reducible $K3$ surfaces have at least one component 
that has dimension at least $15$. As $15 >22-2g$ (the expected dimension of the general fiber) for $g \geqslant 4$, this makes it difficult to prove Theorem \ref{thm:main} directly by semicontinuity around elements of $\V^*_{\delta,\mathfrak{S}_p}$.  We will circumvent this problem in the next section, to prove Theorem \ref{thm:main} for $m=1$.

To this end we start by studying the fibers of the restriction of $\psi'_{\delta}$  to $\W'_{\delta}$.

\begin{prop} \label{prop:unica1'} Let $p\geqslant 3$ and   $g \geqslant 1 $.  Then  the general fiber of  $(\psi'_{\delta})_{|\W'_{\delta}}$  has dimension at most $\max\{0,7-g\}$.
\end{prop}

\begin{proof} We will assume $\delta\geqslant 1$. The case $\delta=0$ is easier, can be treated in the same way and details can be left to the reader.

  The idea of the proof is to define, in Step 1 below, an incidence variety $I_\delta$, described in terms of the geometry of the quadric $\mathcal Q=\PP^1 \x \PP^1$, and a surjective morphism $\theta_\delta: I_\delta \to \mathcal W'_\delta$, with fibers isomorphic to $({\rm Aut}(\mathbb P^ 1))^ 2$.  Rather than studying the moduli map $(\psi'_\delta)_{|\mathcal W'_\delta}$, we will study its composition $\sigma_\delta$ with $\theta_\delta$, which can be described in terms of the geometry of $\mathcal Q$. In Step 2 we will study the fibers of $\sigma_\delta$.

With $\mathcal Q=\PP^1 \x \PP^1$, let $\pi_i:\mathcal  Q \to \PP^1$ for $i=1,2$ be the projections onto the factors. Define
\[ 
{\mathcal Q}^2_{\delta}:= \begin{cases}
\Sym^2(\mathcal Q)  & \mbox{if $\delta$ is odd}, \\
{\mathcal Q}^2 & \mbox{if $\delta$ is even} 
\end{cases}
\]
and denote its elements by 
\[
(P,Q)_{\delta}:= \begin{cases}
P+Q \in \Sym^2(\mathcal Q)  & \mbox{if $\delta$ is odd}, \\
(P,Q) \in {\mathcal Q}^2 & \mbox{if $\delta$ is even.} 
\end{cases}
\]\medskip

\noindent {\bf Step 1: A basic construction}.  Consider the locally closed incidence scheme 
\[
I_{\delta} \subset |-K_{\mathcal Q}| \x \Sym^{g}(\mathcal Q) \x 
{\mathcal Q}^2_{\delta} \]
formed by all triples $(E, Q_1+ \cdots+ Q_{g}, (P,Q)_{\delta})$, with  $E\in |-K_{\mathcal Q}|$, $Q_1,\ldots,Q_{g}\in E \subset \mathcal Q$, $P,Q \in E$, 
such that:\\
\begin{inparaenum} [$\bullet$]
\item  $E$ is smooth;\\
\item  $Q= \phi_{\delta,E}(P)$ (equivalently, $P= \phi_{\delta,E}(Q)$  for odd $\delta$);\\ 
\item $\pi_i(Q_1+\cdots+ Q_{g})$ is reduced (in particular $Q_1,\ldots, Q_g$ are distinct);\\
\item  $\pi_i(D_{\delta,E}(P))$ and $\pi_i(Q_1+\cdots+ Q_{g})$ have no points in common, for $i=1,2$.\\
\end{inparaenum}
The divisor $D_{\delta,E}(P)$  is the one in Definition \ref{def:nonpiaceraaconcy} with $L_i$ defined by the projection ${\pi_i}_{|E}$, for $i=1,2$. Recall that $D_{\delta,E}(P)=D_{\delta,E}(Q)$ when $\delta$ is odd. 

The projection of $I_{\delta}$ to 
$|-K_{\mathcal Q}|$ is dominant, with general fiber of dimension $g+1$, hence
\[\dim( I_{\delta})=g+9.\]

Next we define a surjective map
\begin{equation*} \label{eq:tetta}
 \vartheta_{\delta}: I_{\delta} \rightarrow \W'_{\delta}
\end{equation*}
as follows: let
$\xi=(E, Q_1+ \cdots+ Q_{g}, (P,Q)_{\delta}) \in I_{\delta}$ and set $D_\xi=D_{\delta,E}(P)+Q_1+ \cdots+ Q_{g}$, which is a degree $p+1$ divisor on $E$. Then $\O_E(D_\xi)$  defines, up to projective transformations, an embedding of $E$   as an elliptic normal curve $E'\subset \PP^p$. The degree 2 line bundles $L_i$ defined by the projections ${\pi_i}_{|E}$ define rational normal scrolls $R'_i \subset \PP^p$,  for $i=1,2$, hence the surface $R'=R'_1\cup R'_2$. The hyperplane section of $R'$ cutting out $D_\xi$ on $E$ defines a curve $C_\xi' \in W_{\delta}(R')$. One defines $\vartheta_\delta(\xi)=C'_\xi$. 

The map $\vartheta_\delta$ is surjective. 
Take a curve $C'$ on $R'=  R'_1 \cup R'_2$ corresponding to a point of $\W'_{\delta}$ and consider $E=E'= R'_1 \cap R'_2$ together with the $2:1$ projection maps $f_i: E \to \gamma_i$, where $\gamma_i$ is  the residual curve to the line chain of $C'$ on $R'_i$, cf. 
Definition \ref{def:compopartenza}. Up to choosing an isomorphism $\gamma_i \cong \PP^1$, the map $f_1 \x f_2$ embeds $E$ as an anticanonical curve in $\mathcal Q$. Let $Q_1, \ldots, Q_g$ be the unordered images of the intersection points $\gamma_1 \cap \gamma_2 \subset E$ and $(P,Q)_{\delta}$ be the pair of distinguished points, ordered if $\delta$ is even, and unordered if $\delta$ is odd. This defines a point $\xi=(E, Q_1+ \cdots+ Q_{g}, (P,Q)_{\delta}) \in I_{\delta}$ such that $\vartheta_\delta(\xi)=C'$. This argument shows that the fibers of $\vartheta_\delta$ are isomorphic to $({\rm Aut}(\PP^ 1))^ 2$. 

We have the commutative diagram
\[
\xymatrix{
I_{\delta} \ar[drr]_{\sigma_{\delta}} \ar[rr]^{\vartheta_{\delta}} && \W'_{\delta} \ar[d]^{(\psi'_{\delta})_{|\W'_{\delta}}} \\
 &&  \overline{\M}_g
 }
\]
where $\sigma_{\delta}$ is the composition map. It can be directly defined as follows. Take $\xi=(E, Q_1+ \cdots+ Q_{g}, (P,Q)_{\delta}) \in I_{\delta}$. Call $\Gamma_1$ and $\Gamma_2$ the two $\PP^ 1$s such that ${\mathcal Q} =\Gamma_1\times \Gamma_2$. We have the $g$ points $P_{i,j}=\pi_i(Q_j)\in \Gamma_i$, for $i=1,2$ and $1\leqslant j\leqslant g$.  If $\delta$ is even, we get additional points $P_{1,g+1}=\pi_1(P)\in \Gamma_1$ and $P_{2,g+1}=\pi_2(Q)\in \Gamma_2$. If $\delta$ is odd we get two more points  $P_{1,g+1}=\pi_1(P), P_{1,g+2}=\pi_1(Q)$ on $\Gamma_1$. Then $\Gamma=\sigma_{\delta}(\xi)$ is obtained (recalling Figure \ref{fig:dis}) in one of the following ways:\\
\begin{inparaenum} [$\bullet$]
\item  When $\delta$ is even,  by gluing $\Gamma_1$ and $\Gamma_2$ with the identification of  $P_{1,j}$ with $P_{2,j}$, for $1\leqslant j\leqslant g+1$; in the case $g=1$, we furthermore contract $\Gamma_2$, thus identifying $P_{1,1}$ and $P_{1,2}$ on $\Gamma_1$.\\
\item   When $\delta$ is odd,  by first identifying $P_{1,g+1}$ and $P_{1,g+2}$ on $\Gamma_1$ and then gluing with $\Gamma_2$ by identifying the points  $P_{1,j}$ with $P_{2,j}$, for $1\leqslant j\leqslant g$;
 in the cases $g \leqslant 2$, we furthermore contract $\Gamma_2$ (when $g=2$, this  identifies $P_{1,1}$ and $P_{1,2}$ on $\Gamma_1$, whereas if $g=1$, the curve $\Gamma_2$ is contracted to a smooth point and no further node is created). 
\end{inparaenum}\\

\noindent {\bf Step 2: The fiber.}  By Step 1, to understand the general fiber of  $(\psi'_{\delta})_{|\W'_{\delta}}$ , it suffices to understand the general fiber of $\sigma_{\delta}$. This is what we do next. 

Consider a general element $\Gamma$ in  $\psi'_{\delta}(\W'_{\delta})$.  
Then $\Gamma= X_1 \cup X_2$ is a union of two rational components  when $g \geqslant 3$ or $g=2$ and $\delta$ is even (recall Figure \ref{fig:dis}). If $g=2$ and $\delta$ is odd, then $\Gamma$ is a rational curve with two nodes and we will substitute $\Gamma$ with the stably equivalent curve $X_1 \cup X_2$ where $X_1$ is the normalization of $\Gamma$ at one of the nodes and $X_2\cong \PP^1$ is attached to it at the two inverse images of the node. (The dimensional count that follows does not depend on the choice of the
node on $\Gamma_1$ and the two matching points on $X_2$, as we will show
below). If $g=1$, then $\Gamma$ is a rational curve with one node. When $\delta$ is even, we will substitute $\Gamma$ with the stably equivalent curve $X_1 \cup X_2$ where 
$X_1 \cong \PP^1$ is the normalization of $\Gamma$ and $X_2\cong \PP^1$ is attached to it at the two inverse images of the node. When $\delta$ is odd, we will substitute $\Gamma$ with the stably equivalent curve $X_1 \cup X_2$
where $X_1 \cong \Gamma$ and $X_2 \cong \PP^1$ is attached to an arbitrary smooth point of $X_1$. (The choice of the matching points on $X_1$ and $X_2$ will again not
influence the dimensional count that follows.) Thus, after the substitution, the curve $\Gamma$ looks exactly as one of the curves on the right in Figure \ref{fig:dis}.

Denote by $\Gamma_1$ and $\Gamma_2$ the  normalizations of $X_1$ and $X_2$, respectively,  both isomorphic to $\PP^ 1$, so that $\Gamma_1\times \Gamma_2\cong \mathcal Q$, up to the action of $({\rm Aut}(\mathbb P^ 1))^ 2$ (which will contribute to the fiber of $\sigma_\delta$).  
All curves $X_1\cup X_2$
stably equivalent to
 $\Gamma$ constructed  above for $g=1,2$ and $\delta$ odd and
$g=1$ and
$\delta$ even are equivalent by this action.

If $\delta$ is odd, then $X_2=\Gamma_2 \cong \PP^1$, whereas $X_1$ has one node. On $\Gamma_1\times \Gamma_2\cong \mathcal Q$ we have the $g$ points $Q_j=(x_j,y_j)$, for $1\leqslant j\leqslant g$, which are identified in $\Gamma$. In addition, we have the divisor $z_1+z_2 \in \Sym^2(\Gamma_1) \cong  \Sym^2(\PP^1)$ over the node of $X_1$. Then the $\sigma_{\delta}$--fiber of $\Gamma$ consists  of all $(E, Q_1+ \cdots+ Q_{g}, (P,Q)_{\delta}) \in I_{\delta}$
such that $E$ is a smooth curve in $|-K_{\mathcal Q} \* \I_{Q_1 \cup \cdots \cup Q_g}|$ satisfying the additional condition
\begin{equation}
  \label{eq:condodd}
  Q=\phi_{\delta, E}(P) \; \mbox{and} \; \pi_1(P)+\pi_1(Q) = z_1+z_2.
\end{equation}

If $\delta$ is even, then
$X_i \cong \Gamma_1 \cong \PP^1$, for $i=1,2$,  and the intersection points between the two components yield $g+1$  points $Q_j=(x_j,y_j) \in \mathcal Q$, for $1\leqslant j \leqslant g+1$. Then the $\sigma_{\delta}$--fiber of $\Gamma$ is the union of $g+1$ components obtained in the following way.   Choose an index $j\in \{1,\ldots,g+1\}$, e.g., $j=g+1$. Then the corresponding component of the fiber consists of all $(E, Q_1+ \cdots+ Q_{g}, (P,Q)_{\delta}) \in I_{\delta}$
such that $E$ is a smooth curve in $|-K_{\mathcal Q} \* \I_{Q_1 \cup \cdots \cup Q_g}|$ satisfying the additional condition
\begin{equation}
  \label{eq:condeven}
  Q=\phi_{\delta, E}(P) \; \mbox{and} \; \pi_1(P)=x_{g+1}, \pi_2(Q)=y_{g+1}. 
\end{equation}

Either way, up to the action of $({\rm Aut}(\mathbb P^1))^ 2$, the $\sigma_\delta$--fiber of $\Gamma$ is contained in $|-K_{\mathcal Q} \* \I_{Q_1 \cup \cdots \cup Q_g}|$, and, by assumption, we know that there are smooth curves $E\in |-K_{\mathcal Q} \* \I_{Q_1 \cup \cdots \cup Q_g}|$. 
Hence $\dim(|-K_{\mathcal Q} \* \I_{Q_1 \cup \cdots \cup Q_g}|)=\max \{0, 8-g\}$, which proves the assertion if $g\geqslant 8$.   (This is clear if $g \geqslant 9$; if $g=8$, the projection $I_{\delta} \to \Sym^8(\mathcal Q)$ is dominant, whence we may assume that $Q_1, \ldots, Q_8$ are general points in $\mathcal Q$.)

To finish the proof, we have to prove that for $g \leqslant 7$ and for a general choice of $(E, Q_1+ \cdots+ Q_{g}, (P,Q)_{\delta}) \in I_{\delta}$,  condition \eqref{eq:condodd} or \eqref{eq:condeven} defines a proper, closed subscheme of $|-K_{\mathcal Q} \* \I_{Q_1 \cup \cdots \cup Q_g}|$.

Let $\delta$ be odd. Let us fix a general point $P\in \mathcal Q$ and let us take a general curve $E\in |-K_{\mathcal Q} \otimes \I_P|$. Consider the point $Q=\phi_{\delta,E}(P)$. Fix then another general $E'\in |-K_{\mathcal Q} \otimes \I_P|$, which neither contains $Q$ nor its conjugate on $E$ via $L_1$. The curves $E$ and $E'$ intersect at $7$ distinct points off $P$. Choose $Q_1,\ldots, Q_g$ among them. Then  $E, E' \in |-K_{\mathcal Q} \* \I_{Q_1 \cup \cdots \cup Q_g}|$, and $E$  satisfies \eqref{eq:condodd} (with $z_1=\pi_1(P), z_2=\pi_1(Q)$), whereas $E'$ does not. 

The proof for $\delta$ even is similar and can be left to the reader. \end{proof}


\begin{cor} \label{cor:unica1}   Let $p \geqslant 3$,  $g \geqslant 1$  and $(R_0,C_0) \in \W_{\delta}$. Set $(R'_0,C'_0):=\pi_{\delta}((R_0,C_0)) \in \W'_{\delta}$ and $\Gamma_0:=\psi_{\delta}^*((R_0,C_0))= \psi'_{\delta}((R'_0,C'_0))$, cf. \eqref{eq:diagrammonefichissimo} and \eqref {eq:diagrammonefichissimo2}.

There is at least one component $V_0$ of
$(\psi^*_{\delta})^{-1}(\Gamma_0)$ containing the whole fiber $\pi^{-1}_{\delta}((R'_0,C'_0))$, and any such $V_0$ satisfies  \\
\begin{inparaenum}
\item[(i)] $V_0 \subseteq \W_{\delta}$; \\
\item[(ii)] $\dim( \pi_{\delta}(V_0)) \leqslant \max\{0,7-g\}$.
\end{inparaenum}
\end{cor}

\begin{proof}  
(a) As mentioned above, the fiber $\pi_{\delta}^{-1}((R'_0,C'_0))$ is isomorphic to a dense open subset of  $|T^1_{R'_0}|$, and it is  contained in $(\psi_{\delta}^*)^{-1}(\Gamma_0) \cap \W_{\delta}$. Hence there is at least one component $V_0$ of
$(\psi^*_{\delta})^{-1}(\Gamma_0)$ containing the whole fiber $\pi^{-1}_{\delta}((R'_0,C'_0))$ and 
  any such component contains some $(R,C)\in \mathcal W_\delta$ satisfying that $R=(R',T)$ for a {\it general} element $T \in  |T^ 1_{R'}|$. 
Due to the generality of $T$, the surface $R$ lies off the closure of the Noether--Lefschetz locus in $\K_p$ (see \cite {clm}). Thus, there can be no 
element $(S,C)\in V_0$ with $S$ smooth, because ${\rm Pic}(S)\cong \mathbb Z$, 
   whereas $C$ is reducible, its class in  moduli being $\Gamma_0$. So, if $(S,C)\in V_0$, then $S\in \mathfrak S_p$. Thus $V_0 \sub \V^*_{\delta} \cap (\phi^*_{\delta})^{-1}(\mathfrak{S}_p)$.

If $\dim (V_0)=0$, there is nothing more to prove. So assume there is a disc $\mathbb D\subset V_0$ parametrizing pairs $(R_t, C_t)$
such that $\psi^*_{\delta}(	R_t, C_t)=\Gamma_0$, with $R_t\in \mathfrak{S}_p$ and $C_t \in V_{\delta}(R_t)$, for all $t \in \Delta$, with $R_0=R$. Consider $R'_t=\pi_p(R_t)\in \R_p$, and the divisor $T_t$ on $E'_t$ along which the modification $\pi_t: R_t\to R'_t$ takes place. 
As $\pi_0(C_0) \cap T_0 =\emptyset$, we may assume that  $\pi_t(C_t) \cap T_t =\emptyset$ for all $t\in \Delta$. Hence 
$C_t \cong \pi_t(C_t) \subset R'_t$, which has $\delta$ marked nodes on the smooth locus of $R'_t$, for all $t\in \Delta$. These nodes lie on lines contained in $\pi_t(C_t)$. As 
$\pi_0(C_0)$ contains precisely one connected chain of lines (i.e., its line chain, see Definition \ref {def:compopartenza}), the same holds for $\pi_t(C_t)$ for all $t\in \Delta$. Hence $\pi_t(C_t) \in W_{\delta}(R'_t)$, so that
$C_t \in W_{\delta}(R_t)$. This proves (i).  Assertion (ii)  then follows from Proposition \ref{prop:unica1'}.
\end{proof}

\section{The moduli map for $m=1$} \label{sec:m1}

In this section we will prove the part of Theorem \ref {thm:main} concerning the case $m=1$.  We circumvent the problem of the ``superabundant'' fibers of $\psi^ *_{\delta}$ remarked in the previous section by passing to (a component of) the universal Severi variety over the Hilbert scheme $\H_p$, which is the pullback of the universal Severi variety $\V^*_{\delta}$ over the locus of smooth, irreducible $K3$ surfaces.

 \begin{prop} \label{prop:unica2} Let $p\geqslant 3$. One has:\\
\begin{inparaenum}
\item [(a)] The map $\psi^ *_{\delta}$ is generically finite for $g\geqslant 15$.\\
\item [(b)] The map $\psi^ *_{\delta}$ is dominant for  $1 \leqslant g \leqslant 7$.
\end{inparaenum}
\end{prop}

\begin{proof}  There exists a dense, open subset $\H^{\circ}_p \subset \H_p$, disjoint from $\R'_p$, such that there is a dominant moduli morphism  
\[
\xymatrix{ \mu_p: \H^{\circ}_p \ar[r] &  \K_p \subset \overline{\K}_p.}
\]

We consider the cartesian diagram
\[
\xymatrix{
\mathfrak V^{\circ}_{\delta} \ar[d]_{\Phi_{\delta}^{\circ}} \ar[rr]^ {\nu_p} &&  \V^*_{\delta} \ar[d]^{\phi^*_{\delta}} \\
 \H^{\circ}_p\ar[rr]_ {\mu_p}  &&  \overline{\K}_p
 }
\]
which defines, in the first column, (a component of) a \emph{universal Severi variety} over $\H^{\circ}_p$. Accordingly, we can then consider the moduli map $\Psi^{\circ}_\delta$ fitting in the commutative diagram
\[
\xymatrix{
\mathfrak V^{\circ}_{\delta} \ar[drr]_{\Psi^{\circ}_{\delta}} \ar[rr]^ {\nu_p} &&  \V^*_{\delta} \ar[d]^{\psi^ *_{\delta}}\\
 &&  \overline{\M}_g
 }
\]
and it suffices to prove that the general fiber of $\Psi^{\circ}_{\delta}$ has dimension  
\[ 
p^2+2p+\max\{0,22-2g\}=p^2+2p+\max\{0,7-g\}+\max\{0,15-g\},
\]
the equality due to the assumptions on $g$.  We prove this by semicontinuity. To this end we introduce a partial compactification $\mathfrak V_{\delta}$ of $\mathfrak V^{\circ}_{\delta}$ fitting into the commutative diagram
\[
\xymatrix{
\mathfrak V^{\circ}_{\delta} \ar[d]_{\Phi^{\circ}_{\delta}} \ar@{^{(}->}[r] &  \mathfrak V_{\delta} \ar[d]^{\Phi_{\delta}} \\
\H^{\circ}_p \ar@{^{(}->}[r] & \H_p
 }
\]
The  fiber of $\Phi_\delta$ over any element $R' \in \R'_p \subset \H_p$ contains the varieties $W_{\delta}(R')$ by Proposition \ref{prop:rightdim} (and Definition \ref{def:componente}). Possibly after substituting $\mathfrak V_{\delta}$ with a dense, open subset (still containing the varieties $W_{\delta}(R')$), we
have a natural moduli map $\Psi_{\delta}: \mathfrak V_{\delta} \to \overline{\M}_g$ extending $\Psi^{\circ}_{\delta}$. 

Take general $R' \in \R'_p$ and $C'_0 \in W_{\delta}(R')$ and set 
$\Gamma:=\Psi_{\delta}((R',C'_0))=\psi'_{\delta}((R',C'_0))$, cf. \eqref{eq:diagrammonefichissimo} and \eqref{eq:diagrammonefichissimo2}. By semicontinuity,  it suffices to prove that any component $F_0$ of 
$\Psi_{\delta}^{-1}(\Gamma)$ passing through $(R',C'_0)$ has dimension at most $p^2+2p+\max\{0,7-g\}+\max\{0,15-g\}$. This will be done by first bounding the dimension of $F_0\cap \Phi_\delta^ {-1}( \R'_p)$ and then bounding the dimension of all possible deformations of $(R',C'_0)$ in $F_0$ outside of $\Phi_\delta^ {-1}( \R'_p)$. We make this more precise as follows.

Let $\mathfrak K$ be the reduced affine tangent  cone to $F_0$ at  $(R', C'_0)$.
Consider the linear map
 $$
\xymatrix{
L:T_{F_0,(R', C'_0)}\ar[r]& T_{\H_p, R'}\ar[r] &H^0(T^1_{R'}),
}
$$
defined by composing the differential $d_{(R',C'_0)}(\Phi_\delta|_{F_0})$ of $\Phi_\delta$ at $(R',C'_0)$ with the map \eqref{standard},
and its restriction $L_{|_{\mathfrak K}}$ to $\mathfrak K$. By standard deformation theory, $L_{|_{\mathfrak K}}^{-1}(0)$ coincides with the affine tangent
cone to $F_0\cap \Phi_\delta^ {-1}( \R'_p)$ at $(R', C'_0)$. Therefore, 
\[
\dim (F_0)=\dim (\mathfrak K) \leqslant \dim(F_0\cap \Phi_\delta^ {-1}(\mathcal R'_p))+\dim(L(\mathfrak K))
\]

\noindent and we are done once we prove that 
\begin{equation}\label{eq:int} \dim(F_0\cap \Phi_\delta^ {-1}( \R'_p))\leqslant p^2+2p+\max\{0,7-g\}
\end{equation}
and
\begin{equation}\label{eq:int2}
\dim ( L(\mathfrak K')) \leqslant \max\{0,15-g\} \; \; 
\mbox{for any irreducible component $\mathfrak K'$ of $\mathfrak K$.}
\end{equation}
\medskip

\noindent {\bf Proof of \eqref{eq:int}:}   Under the $\PGL(p+1,\CC)$-quotient map\[
\xymatrix{
\Phi_{\delta}^{-1} (\R'_p)\ar[r] \ar[d]^{\Phi_{\delta}} &  \V^*_{\delta,\R_p} \ar[d]\\
 \R'_p \ar[r] & \R_p
 }
\] 
the variety  $F_0\cap \Phi_\delta^ {-1}( \R'_p)$ is mapped with fibers of dimension $p^2+2p$ to a variety $V'_0 \sub (\psi'_{\delta})^{-1}(\Gamma) \sub \V^*_{{\delta},\R_p}$ intersecting $\W'_{\delta}$. The component $V_0$ of {$(\psi^*_{\delta})^{-1}(\Gamma)\cap \mathcal V^ *_{\delta,\mathfrak S_p}$} containing $(\pi_{\delta})^{-1}(V'_0)$ (cf. \eqref{eq:diagrammonefichissimo} and \eqref{eq:diagrammonefichissimo2}) lies in $\W_{\delta}$ and satisfies  $\dim (\pi_{\delta}(V_0))\leqslant \{0, 7-g\}$ by Corollary \ref{cor:unica1}. Therefore   $\dim (V'_0)\leqslant \{0, 7-g\}$ by the surjectivity of $\pi_{\delta}$,  proving \eqref {eq:int}.\medskip

\noindent {\bf Proof of \eqref{eq:int2}:}  
We may assume that  $F_0$ contains a one--parameter family containing $(R',C'_0)$  and such that its general point is a pair $(S,C)$ with $S \not \in \R'_p$, otherwise $F_0=F_0\cap \Phi_\delta^{-1} (\R'_p)$ and there is nothing more to prove.

  We claim that $S$ has at most isolated singularities. Indeed, suppose $S$ is singular along a curve $\Sigma$  and let $\mathcal S$ be the total space of the family. Then the flat limit of $\Sigma$ on the central fiber sits in the singular locus scheme of $R'$, which is $E'$. Since $E'$ is a general elliptic normal curve (hence it is smooth and irreducible), then also $\Sigma$ is an elliptic normal curve and $S$, as well as $R'$, has simple normal crossings along $\Sigma$.  Then  $\mathcal S$ has normal crossings along the surface  which is the total space of the family of singular curves of the surfaces $S$. Let $\mathcal S'$ be the normalization of $\mathcal S$. The central fiber of $\mathcal S'$ consists of two smooth connected components,  i.e., the proper transforms of $R'_1$ and $R'_2$. This implies that the general fiber  of $\mathcal S'$ is also disconnected, consisting of two smooth components isomorphic to $R'_1$ and $R'_2$. Therefore the general fiber  $S$ of $\mathcal S$ sits in $\mathcal R'_p$, a contradiction, which proves the claim.

Since $S$ has  (at most)  isolated singularities and degenerates to $R'$, it must have  at worst $A_n$--singularities, so it is a projective model of a smooth $K3$ surface. Thus we may find a one--dimensional family 
$(S_t, C_t)_{t\in \mathbb D}$ of points in $\V^ *_\delta$ parametrized by the disc $\mathbb D$, such that 
$R:=S_0 \in \mathfrak{S}_p$  and $\pi_p(R)=R'$ (that is, $R$  
is a birational modification of $R'$), 
$C_0$ is the total transform of a $C'_0$ on $R'$, $\psi^*_{\delta}(S_t, C_t)=\Gamma$ for all $t \in \mathbb D$, and $S_t$ is smooth for  all $t \in \mathbb D-\{0\}$.

Since $\Gamma$ has $g+1$ nodes, the curves $C_t$ (including $C_0$) have at least $g+1$ {\it unmarked} nodes mapping to the nodes of $\Gamma$ through the process of partial normalization at the $\delta$ marked nodes and possibly semi-stable reduction. On $C_0$ they must lie on the smooth locus of $R$.  (Indeed, it is well-known, and can be verified by a local computation, that a node of a limit curve lying on $E$ smooths as $R$ smooths, cf. e.g. \cite{C}, \cite[\S 2]{galati} or \cite[Pf. of Lemma 3.4]{galknu}.) This means that these nodes of $C_0$ lie on the exceptional curves of the morphism $R \to R'$. 
Therefore  the $g$ intersection points $\gamma_1\cap \gamma_2$ (all lying on the elliptic curve $E'$) plus one (at least) among the intersections of the line cycle of $C_0'$ with the elliptic curve $E'$ are contained in  the divisor $T\in |T^ 1_{R'}|$. Hence, we have $g+1\leqslant \deg(W)=16$, 
 and if $g \leqslant 14$, we  have  
$\dim  (L(\mathfrak K')) \leqslant h^0(T^1_{R'})-(g+1)=15-g$ for any irreducible component $\mathfrak K'$ of $\mathfrak K$, proving 
\eqref{eq:int2}. Therefore, we have left to prove that the case
$g=15$ cannot happen. 

To this end, assume $g=15$. The  above argument shows that the $g$ intersection points $\gamma_1\cap \gamma_2$ plus a point $x=\phi_{k,E'}(P)$ are the zeroes of a section of $T^ 1_{R'}$, for some integer $0 \leqslant k \leqslant \delta$ (cf. Definition \ref{def:nonpiaceraaconcy}). 

If $\delta$ is odd,  the $g$ intersection points $\gamma_1\cap \gamma_2$ are in  $|L_{E'}\otimes (L_2^*)^ {\otimes \frac{\delta+1}{2}}|$. It follows that 
\[ \O_{E'}(x) \cong  T^ 1_{R'} \otimes L^*_{E'}\otimes L_2^{\otimes \frac{\delta+1}{2}}\]
which implies that $x$ is uniquely determined, hence also the line chain of $C'_0$ is, contradicting the fact that $C'_0$ is general
in $W_{\delta}(R')$. 

If $\delta$ is even,  the $g$ intersection points $\gamma_1\cap \gamma_2$ are in 
\[|L_{E'} \otimes \O_{E'}(-P) \otimes (L_1^*)^ {\otimes \frac{\delta}{2}}|=
|L_{E'} \otimes \O_{E'}(-Q) \otimes (L_2^*)^ {\otimes \frac{\delta}{2}}|.\] 
Hence 
\begin{equation}
  \label{eq:relcazzo}
\O_{E'}(x) \cong   T^ 1_{R'} \otimes L_{E'}^ *\otimes \O_{E'}(P)\otimes L_1^ {\otimes \frac{\delta}{2}}\cong T^ 1_{R'} \otimes L_{E'}^ *\otimes \O_{E'}(Q)\otimes L_2^ {\otimes \frac{\delta}{2}} .\end{equation}

If $k$ is odd,  combining \eqref{eq:relgen-odd} (with  $y$ replaced by $P$) with \eqref{eq:relcazzo} yields 
\[ \O_{E'}(2P) \cong L_{E'}\otimes (T^ 1_{R'})^* \otimes L_2^ {\otimes \frac{k+1}{2}}\otimes (L_1^*)^{\otimes \frac{k+\delta-1}{2}},\]
which implies only finitely many choices for $P$, and we get a contradiction as before. 
If $k$ is even, then \eqref{eq:relgen-even}, with $x$ and $y$ replaced by $P$ and $x$ respectively, combined with \eqref{eq:relcazzo} yields 
\[ L_{E'} \simeq T^ 1_{R'} \otimes L_2^ {\otimes  \frac{k}{2}}\otimes L_1^ {\otimes \frac{\delta-k}{2}},\]
which, together with \eqref {eq:B} gives a non-trivial relation between $L_{E'}, L_1$ and $L_2$, a contradiction. 
\end{proof}

At this point, the proof of Theorem \ref{thm:main} for $m=1$ is complete  (noting that the assertion is trivial for $g=0$).

\section{More limits of nodal curves on reducible $K3$ surfaces} \label{sec:limnod}

In order to treat the case $m>1$ 
we  need to consider more limits of nodal curves on surfaces in $\mathfrak{S}_p$. They are similar to the ones constructed by X. Chen in \cite[\S 3.2]{C}. 

\begin{definition} \label{def:curvedichen}
For each $m \geqslant 1$, $R \in \mathfrak{S}_p$    and the corresponding $R':=\pi_p(R)  \in \R_p$,   we define $U'_m(R') \subset |mL_{R'}|$ to be the  locally closed subset of    curves of the form
\[ C'=C^1_1 \cup C^1_2 \cup \cdots \cup C^1_{m-1} \cup C^1_m \cup C^2_1 \cup C^2_2 \cup \cdots \cup C^2_{m-1} \cup C^2_m\]
where:\\
\begin{inparaenum}[$\bullet$]
\item $C^i_j \subset R'_i$ for $i=1,2$ and $1 \leqslant j \leqslant m-1$;\\
\item $C^i_j \in |\sigma_i|$ for $1 \leqslant j \leqslant m-1$ and $C^i_m \in |\sigma_i+mlF_i|$ if $p=2l+1$ is odd; \\
\item $C^i_j \in |\sigma_i+F_i|$ for $1 \leqslant j \leqslant m-1$ and $C^i_m \in |\sigma_i+(ml-m+1)F_i|$ if $p=2l$ is even; \\
\item there are points $P, Q_0, \ldots, Q_{2m}, Q$ on $E'$ such that
  \begin{eqnarray*}
   C^1_j \cap E'=Q_{2j-2}+Q_{2j-1}, & C^2_j \cap E'=Q_{2j-1}+Q_{2j}, & \text{for} \quad 1 \leqslant j \leqslant m-1 \\
 C^1_m \cap E'=2mlQ +P+ Q_{2m-2}, & C^2_m \cap E'=2mlQ+P+Q_0, & 
  \end{eqnarray*}
if $p$ is odd, and
  \begin{eqnarray*}
   C^1_j \cap E'=Q_{2j-2}+2Q_{2j-1}, & C^2_j \cap E'=2Q_{2j-1}+Q_{2j}, & 1 \leqslant j \leqslant m -1 \\
 C^1_m \cap E'=(2ml-2m+1)Q +P+ Q_{2m-2}, & C^2_m \cap E'=(2ml-2m+1)Q+P+Q_0, & 
  \end{eqnarray*}
if $p$ is even.
\end{inparaenum}

We denote by $U_m(R) \subset |mL_R|$ the set of total transforms of the curves in $U'_m(R')$ passing through two points of $T$ different from $Q$. 
\end{definition}

 We remark that the curves in $U'_m(R')$ (or $U_m(R)$) are tacnodal at $Q$ and, for even $p$, also at $Q_i$ with $i$ odd. 
Figure \ref{fig:dischen} depicts a member of $U'_3(R')$ for odd $p$. The intersections on $E'$ (which are not all shown in the figure) are marked with dots and are all transversal except at $Q$. The case of  $p$ even is similar, except that the intersections at $Q_1$ and $Q_3$ are tangential. 

\begin{figure}[ht] 
\[
\includegraphics[width=4cm]{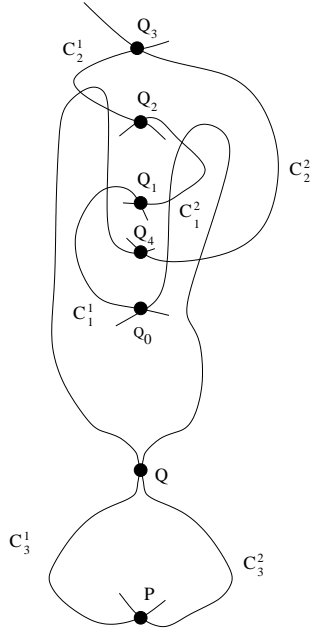} 
\] 
\label{fig:dischen}
\caption{Sketch of a member of $U'_3(R')$ when $p$ is odd. }
\end{figure}

\begin {lemma}\label{lem:clear} We have $\dim (U'_m(R'))=2$ and $\dim (U_m(R))=0$.
\end{lemma}

\begin{proof}
It is easily seen that any of the three maps $U'_m(R')\to {E'}^2$ mapping $C'$ to $(Q,P)$,$(Q,Q_0)$ and $(Q,Q_{2m-2})$, respectively, is an isomorphism, proving the lemma. 
\end{proof}

\begin{proposition} \label{prop:curvedichen}
 For any $m \geqslant 1$, the curves in $U_m(R)$ deform to a set $U_m(S)$ of rational {\it nodal} curves in $|mL|$
on the general $[S,L] \in \K_p$. 

More precisely, there is a $19$-dimensional subvariety $\U_m \subset \V_{m,p(m)}$ with a partial compactification $\overline{\U}_m \to \overline{\K}_p$ whose fiber over $R \in \mathfrak{S}_p$ is
$U_m(R)$ and whose fiber $U_m(S)$ over a general $S \in \K_p$ is a nonempty finite subset of  $V_{m,p(m)}(S)$. 
\end{proposition}

\begin{proof}
  This follows the lines of proof in \cite[\S 3.2]{C} or \cite[Proof of Thm.~1.1]{galknu}. For the reader's convenience, we outline here the main ideas, without dwelling on details.

Take a curve $C \in U_m(R)$ that is the total transform of a general curve $C' \in U'_m(R')$. 

If $p=2l+1$ is odd, then $C$ has a total of $2lm(m-1)+2$ nodes
on the smooth locus of $R$:\\
\begin{inparaenum} [$\bullet$]
\item two nodes occur on the two exceptional curves of $\mathfrak e$ that $C$ contains;\\
\item  the remaining nodes are the points
$C^i_j \cap C^i_m$ for $i=1,2$, $1\leqslant j\leqslant m-1$, a total of $2(m-1)$ times 
\[ C^i_j\cdot C^i_m=\sigma_i\cdot (\sigma_i+mlF_i)=ml.\]
\end{inparaenum}
 
Moreover, the point $Q$ on $C$ is a $2ml$--tacnode, which can be (locally) deformed to $2ml-1$ nodes for a general deformation of $R$ to a $K3$ surface  (see \cite {ran},
\cite[Sect. 2.4]{ch} or \cite[Thm 3.1]{galknu} for a generalization of this result). Since, by construction, $C$ does not contain subcurves lying in $|hL_R|$ for any $1 \leqslant h \leqslant m-1$ (this is like saying that the aforementioned singularities of $C$ are non--disconnecting), one checks that, for a general deformation of $[R,L_R]$ to a $K3$ surface $[S,L]$, the curve $C$ deforms to an irreducible, nodal curve in $|mL|$  with a total of
\[ \Big[2nm(m-1)+2\Big] + \Big[2mn-1\Big]=2nm^2+1=p(m)\]
nodes, as claimed.

The case $p$ even is similar and can be left to the reader. \end{proof}

\section{Dominance of the moduli map for $m \geqslant 2$}\label{sec:dominance}
In this section we will use  Proposition \ref{prop:unica2}  and the curves in Definition \ref{def:curvedichen} to prove part $(A)$ of Theorem \ref {thm:main} for $m\geqslant 2$.  We assume $p\geqslant 3$    in the whole section.   

The following Proposition fixes a gap in \cite [Prop. 1.2]{ck}, whose proof is incomplete, cf. \cite{ckerr}:

\begin{prop} \label{lemma:solok3}
Let $S \in \K_p$ be general  and $m \geqslant 1$ be an integer. Let $V$ be a fully complete component of $V_{m,\delta}(S)$. Then the moduli map 
\[
{\psi_{m,\delta}}_{|V}: V\to \overline \M_{g_{m,\delta}}
\]
is generically finite. In particular this applies to  any  component of  $V_{\delta}^ *(S)$.
\end{prop}

\begin{proof} By fully completeness,  we have the filtration
\[
 \emptyset \neq V_{[m,\delta,{ p(m)}]}(S)
 \cap V  \subset V_{[m,\delta,{ p(m)}-1]}(S) \cap V  \subset \ldots \subset V_{[m,\delta,\delta+1]}(S) \cap V   \subset  V  \]
Set $V_i={\psi_{m,\delta}}(V_{[m,\delta, { p(m)}-i] }(S) \cap V )$, for $0\leqslant i\leqslant {g_{m,\delta}}-1$, and 
$V_{g_{m,\delta}}=\psi_{m,\delta}(V)$. Then we  have
\[
 \emptyset \neq  V_0\subsetneq V_1\subsetneq \ldots \subsetneq V_{g_{m,\delta}-1}\subsetneq V_{g_{m,\delta}},
\]  where $\overline V_i$ is a proper closed subvariety of $\overline V_{i-1}$ for all $i=0,\ldots, g_{m,\delta}-1$.

Since $\dim (V_0)=0$, we have $\dim (V_{g_{m,\delta}})\geqslant g_{m,\delta}$, which yields the first assertion. The final assertion follows by Remark \ref {rem:imp}. 
\end{proof}

\begin{remark} \label{rem:generibassi} The existence of fully complete components $V$ of $V_{m,\delta}(S)$ for all $m\geqslant 1$, for a general $S\in \mathcal K_p$,  follows from the existence of irreducible nodal rational curves in $|mL|$ proved in \cite {C} (or by Proposition \ref{prop:curvedichen}). Thus the  general point in $V$ parametrizes an {\it irreducible} curve. 

In particular, Proposition \ref {lemma:solok3} shows  that the moduli map $\psi_{m,p(m)-1}: \V_{m,p(m)-1} \to \M_1$ 
is dominant even when restricted to curves on a single general $K3$, for any $m \geqslant 1$. 
\end{remark}

Next we need a technical construction and  a lemma.  Consider the commutative diagram
\[
\xymatrix{
\V^*_{\delta} \ar[d]_{\phi^*_{\delta}}\ar[drr]^ {\psi^ *_{\delta}} \ar[rr]^ {\widetilde \psi_\delta} && \widetilde  \M_{g_{1,\delta}}\ar[d]^{\varphi_{\delta}} \\
 \overline {\K}_p  &&\overline \M_{g_{1,\delta}}
 }
\]
where the triangle is the Stein factorization of $\psi^ *_\delta$. We will abuse notation and we will  
identify an element $\Gamma\in \widetilde \M_{g_{1,\delta}}$ with its image in $\overline \M_{g_{1,\delta}}$. 

Assume $1 \leqslant {g_{1,\delta}}\leqslant 7$.  By  Proposition \ref{prop:unica2},  given a general $\Gamma\in {\rm Im}(\widetilde \psi_{\delta})$, the fiber $\widetilde \psi_{\delta}^ {-1}(\Gamma)$ is irreducible of dimension
\begin{equation*}\label{eq:fiber}
\dim(\widetilde \psi_{\delta}^ {-1}(\Gamma))=\dim (\V^*_{\delta}) - \dim (\widetilde \M_{{g_{1,\delta}}}) =22-2{g_{1,\delta}}.
\end{equation*}
Set  $T(\Gamma):= \phi^*_{\delta}(\widetilde \psi_{\delta}^ {-1}(\Gamma))$
which is irreducible and,  by Proposition \ref {lemma:solok3},  one has
\begin{equation} \label{eq:fiber'} 
\dim (T(\Gamma))=\dim(\widetilde \psi_{\delta}^ {-1}(\Gamma))=22-2{g_{1,\delta}}.
\end{equation}

 Now we need to introduce the irreducible component $\mathcal V^ \bullet_p$ of $ \overline {\mathcal V}_{1,p}$ containing $\mathcal W_{p-2,2}$, whose fibre over $S\in \mathcal K_p$ we will denote by $V^ \bullet_p(S)$. Consider a general  $S\in \K_p$. For any pair $(s,a)$ of  positive integers, with  $s \leqslant 4$ (hence ${g_{1,\delta}}+s \leqslant 11$), there  are on $S$ finitely many curves of the form
\begin{equation}\label{eq:z}
Z=B_{1}+\cdots+ B_{s-1}+ B_{s},\,\,\, \text{with}\,\,\, B_i\in {V^ \bullet_p(S)},\,\,\, \text{for}\,\,\, 1\leqslant i\leqslant s-1, \,\,\,\text{and}\,\,\, B_s\in U_a(S),
\end{equation}where $U_a(S)$ is as in  Proposition \ref{prop:curvedichen} (cf. Definition\,\ref{def:curvedichen}).

Given a general $(S, C)\in \V^ *_\delta$, set $\Gamma:=\widetilde \psi_\delta(S, C)$, which is the (class of the) normalization $\widetilde C$ of $C$. For all $S'\in T(\Gamma)$, we have a  morphism $\Gamma \to S'$,  whose image we denote by $C'$,  and each of the curves $B_i$ on $S'$ cuts out on 
$C'$ a divisor, which can be pulled--back to a divisor $\mathfrak g_i$ on $\Gamma$, for $1\leqslant i\leqslant s$.

  Let $\mathfrak T \subset ({\rm Sym}^ 2(\Gamma))^ s \times T(\Gamma)$ be the incidence subvariety described by the points 
$$((\mathfrak d_1,\ldots, \mathfrak d_s), S')\in 
({\rm Sym}^ 2(\Gamma))^ s \times T(\Gamma)$$ such that
$\mathfrak d_i\leqslant \mathfrak g_i$, with $\mathfrak g_i$ the divisor cut out by $B_i\subset S'$ on $\Gamma$, for all $i=1,\ldots, s$.  Then let $\mathfrak I$ be the image of $\mathfrak T$ in $({\rm Sym}^ 2(\Gamma))^ s$.

\begin{lemma}\label{lemma:fibra} In  the above setting:\\
\begin{inparaenum}[(i)]
\item $\mathfrak I=({\rm Sym}^ 2(\Gamma))^ s$;\\
\item  for the general $S'\in T(\Gamma)$ the curves of the form $C'+Z$  (with $Z$ as in \eqref{eq:z})  are nodal.
\end{inparaenum}
\end{lemma}

\begin{proof}   We specialize $(S,C)\in \V^ *_\delta$ to $(R_0,C_0) \in \mathfrak \W_{\delta}$, set $\Gamma:=\psi_\delta ^ *(S,C)\in \M_{g_{1,\delta}}$ and $\Gamma_0=\psi^ *_\delta(R_0,C_0)\in \overline \M_{g_{1,\delta}}$, and  prove the assertions (i) and (ii) in the limit situation for $\Gamma_0$ and $T(\Gamma_0)$. 

Each $S'\in T(\Gamma_0)$ contains a curve $C'\in V_\delta^ *(S')$ with $\psi^ *_\delta(C')=\Gamma_0$. 
By the proof of   Proposition \ref{prop:unica2},    
all components of $T(\Gamma_0)=\phi_{\delta}^*(\tilde{\psi}_{\delta}^{-1}(\Gamma_0))$ have dimension at most $22-2{g_{1,\delta}}$. By \eqref{eq:fiber'} we conclude that  {\it all} components of $T(\Gamma_0)$ have dimension precisely $22-2g_{1,\delta}$. In particular, this applies to the component $T_0:=\phi^*_{\delta}(V_0)$, cf. \eqref{eq:diagrammonefichissimo}, where $V_0$ is a component of $(\psi^*_{\delta})^{-1}(\Gamma_0)$ satisfying the conditions in Corollary \ref{cor:unica1}. We recall that $V_0$ contains $(R_0,C_0)$ and satisfies $\dim (\pi_{\delta}(V_0)) \leqslant \max\{0,7-g_{1,\delta}\}$ by Corollary \ref{cor:unica1}(b), and $T_0$ lies in the specialization of $T(\Gamma)=\phi_{\delta}^*(\tilde{\psi}_{\delta}^{-1}(\Gamma))$. Therefore, recalling again \eqref{eq:diagrammonefichissimo}, we have  
\[
\dim (\pi_p(T_0)) = \dim (\pi_p(\phi_{\delta}^*(V_0)) )
\leqslant \dim (\pi_{\delta}(V_0)) \leqslant \max\{0,7-{g_{1,\delta}}\} \]
and  
the general fiber of ${\pi_p}_{|T_0}$ has dimension 
\[ f\geqslant \dim (T_0) -\max\{0,7-{g_{1,\delta}}\} \geqslant 22-2{g_{1,\delta}}-\max\{0,7-{g_{1,\delta}}\} =\min\{22-2{g_{1,\delta}}, 15-{g_{1,\delta}}\} \geqslant 2s \]
(remember the assumptions on $s$ and ${g_{1,\delta}}$).
Then, inside $T_0$  we can find a subscheme 
$T'_0$ of dimension at least $2s$ consisting solely of surfaces that are birational modifications of a fixed $R'$ and it is therefore parametrized by a family $\mathcal X$ (of dimension at least $2s$)  of divisors in $|T^ 1_{R'}|$. For $R$ corresponding to the general point of $\mathcal X$, the curves in  $V^\bullet_p(R)$ are limits of rational nodal curves in $V^ \bullet_p(S')$  for general $S' \in T(\Gamma)$, and similarly the curves in $U_a(R)$ are limits of rational nodal curves in $U_a(S')$ (see Proposition \ref {prop:curvedichen}).   We will prove that (i) and (ii) hold by proving that they hold  in the limiting situation on $R$.

 We first claim that for the general element $(C_0, B'_1, \ldots, B'_{s-1}, B'_s) \in W_{\delta}(R') \x (W_{p-2}(R'))^{s-1} \x U'_a(R')$, the curve 
$C_0+B'_1 + \cdots + B'_{s-1} + B'_s$ is nodal off the tacnodes of $B'_s$. 

To prove the claim,  note that for any $0 \leqslant j \leqslant p-1$, the variety $W_{p-j}(R')$ parametrizes a $j$--dimensional family of curves. Its general element contains a line chain, that 
varies in a $1$--dimensional, base point free system, so the lines of these chains are in general different from lines contained in the general member of 
$U'_a(R')$ and do not cause any non-nodal singularities in $C_0+B'_1 + \cdots + B'_{s-1} + B'_s$. Fixing a line chain $\mathfrak C$, the family of curves in $W_{p-j}(R')$  containing it is a linear system $\L_\mathfrak C$ of dimension $j-1$,  whose general element, minus $\mathfrak C$,  is a smooth rational curve on every $R_i'$, so that $\L_\mathfrak C$ is base point free off $\mathfrak C$ whenever $j \geqslant 2$. Hence, the general curves $B'_1, \ldots, B'_{s-1}$ in $W_{p-2}(R')$ intersect both $C_0$ and $B'_s$ transversely, and if $\delta <p-1$, then also the general $C_0$ in $W_{\delta}(R')$ intersects $B'_s$ transversely. It remains to prove that the latter holds also when $\delta=p-1$, that is, when $g_{1,\delta}=1$. As we saw, non-transversal intersections can only occur off the line chain of $C_0$, that is, on the two components $\gamma_i \subset R'_i$, $i=1,2$,
of $C_0$.  Denote by $C^i_j$, $i=1,2$, $1 \leqslant j \leqslant m$ the components of $B'_s$ as in Definition \ref{def:curvedichen}.

If $p=2l+1$ is odd, then  $R'_i \cong \PP^1 \x \PP^1$ and $\gamma_i \sim C^i_j \sim \sigma_i$, for $i=1,2$ and $1 \leqslant j \leqslant m-1$. 
Hence $\gamma_i$ only intersects the component $C^i_m \in |\sigma_i+mlF_i|$ of $B'_s$. As we vary the curves $B'_s$ with fixed tacnode $Q= C^1_m \cap C^2_m$ (see Definition \ref{def:curvedichen} and Figure \ref{fig:dischen}), each component $C^i_m$ varies freely in the linear system $|(\sigma_i+mlF_i) \* \I_Z|$, where $Z$ is the length--$(ml)$ subscheme of $E'$ supported at $Q$. This linear system is base point free off $Q$. It follows that each $\gamma_i$ intersects the general $C^i_m$ transversely, whence $C_0$ intersects the general $B'_s$ transversely. 

If $p=2l$ is even, then $R'_i \cong \FF_1$, $\gamma_1 \sim\sigma_1+F_1$, $\gamma_2 \sim \sigma_2$ and $C^i_j \sim \sigma_i+F_i$, for $i=1,2$ and $1 \leqslant j \leqslant m-1$. Hence, $\gamma_1 \cdot C^1_j=1$ and
 $\gamma_2 \cdot C^2_j=0$ for $1 \leqslant j \leqslant m-1$ and we again only need to check that $\gamma_i$ intersects $C^i_m \in |\sigma_i+(ml-m+1)F_i|$ transversely, which can be proved exactly as in the case $p$ odd.

This finishes the proof of the claim. 

Since  the curves $B'_i$ move in an at least two-dimensional   family,  
for $1\leqslant i\leqslant s$,  the curves  $Z_0':= B'_1+ \cdots + B'_{s-1}+ B'_s$  cuts out on $C_0$ an algebraic system $\mathfrak G$ and the associated variety $\mathfrak I'$ equals $(\Sym^2(C_0))^s$. 
Hence, for the general $s$--tuple of pairs of points of $C_0$, we may find a $(B'_1, \ldots, B'_{s-1}, B'_s) \in (W_{p-2}(R'))^{s-1} \x U'_a(R')$ such that $Z'_0$ contains the given points on $C_0$ and is 
nodal off the tacnodes of $B'_s$. 
Each  $B'_i$ in $Z'_0$ cuts $E'$ in  $p+1$ points for $1\leqslant i\leqslant s-1$ and in at least 3 points off the tacnodes if $i=s$. Since $\dim(\mathcal X)\geqslant 2s$, we can find a surface $R$ corresponding to a point of $\mathcal X$ for which the modification $R\to R'$ involves two of the intersection points of each $B'_i$ with $E'$ for $1\leqslant i\leqslant s$ (off the tacnodes if $i=s$). We denote by $Z_0$ the total transform of $Z'_0$ on $R$.
Then  $R \in T_0$, and we denote by $B_{0,i}$ the total transform of each $B'_i$ on $R$, for $1\leqslant i\leqslant s$. Then such a curve is the limit of a nodal rational curve in $V^ \bullet_p(S')$ for $1\leqslant i\leqslant s-1$, and $B_{0,s}$ is the limit of a nodal rational curve in $U_a(S')$. 

 Thus, we have proved that  (i) and (ii) hold in the limiting situation for $R$, $C_0$ and $Z_0$. Hence,  they also hold in general.
\end{proof}

Let $s \leqslant 4$ and $a$ be positive integers. Then for  $[S,L]\in \K_p$  general we can consider the locally closed subset $V_{(\delta,s,a)}(S)$ in $|(s+a)L|$ consisting of all nodal curves of the form $C+Z$, with $(S, C)\in V_\delta^ *(S)$ and $Z$ as in \eqref {eq:z}. Note that $V_{(\delta,s,a)}(S)$ is non--empty for $[S,L]$ general by Lemma \ref {lemma:fibra} and $\dim(V_{(\delta,s,a)}(S))={g_{1,\delta}}$. In the same way, we have the universal family $\V_{(\delta, s,a)} \to \K_p$, with fiber $V_{(\delta,s,a)}(S)$ over $S$.

 Marking all the nodes of a curve $C+Z$ but two in the divisor $\mathfrak g_i$ cut out by $B_i$ on $C$, for  $1\leqslant i\leqslant s$,  one checks that the normalizations are $2$-connected, hence  these curves lie in $V_{m,\zeta}(S)$, where $m=s+a$ and $\zeta=p(m)-{g_{1,\delta}}-s$, with stable models in $\overline{\M}_{{g_{1,\delta}}+s}$ as shown in Figure \ref{fig:dis2} below.

\begin{figure}[ht] 
\[\
\includegraphics[width=8cm]{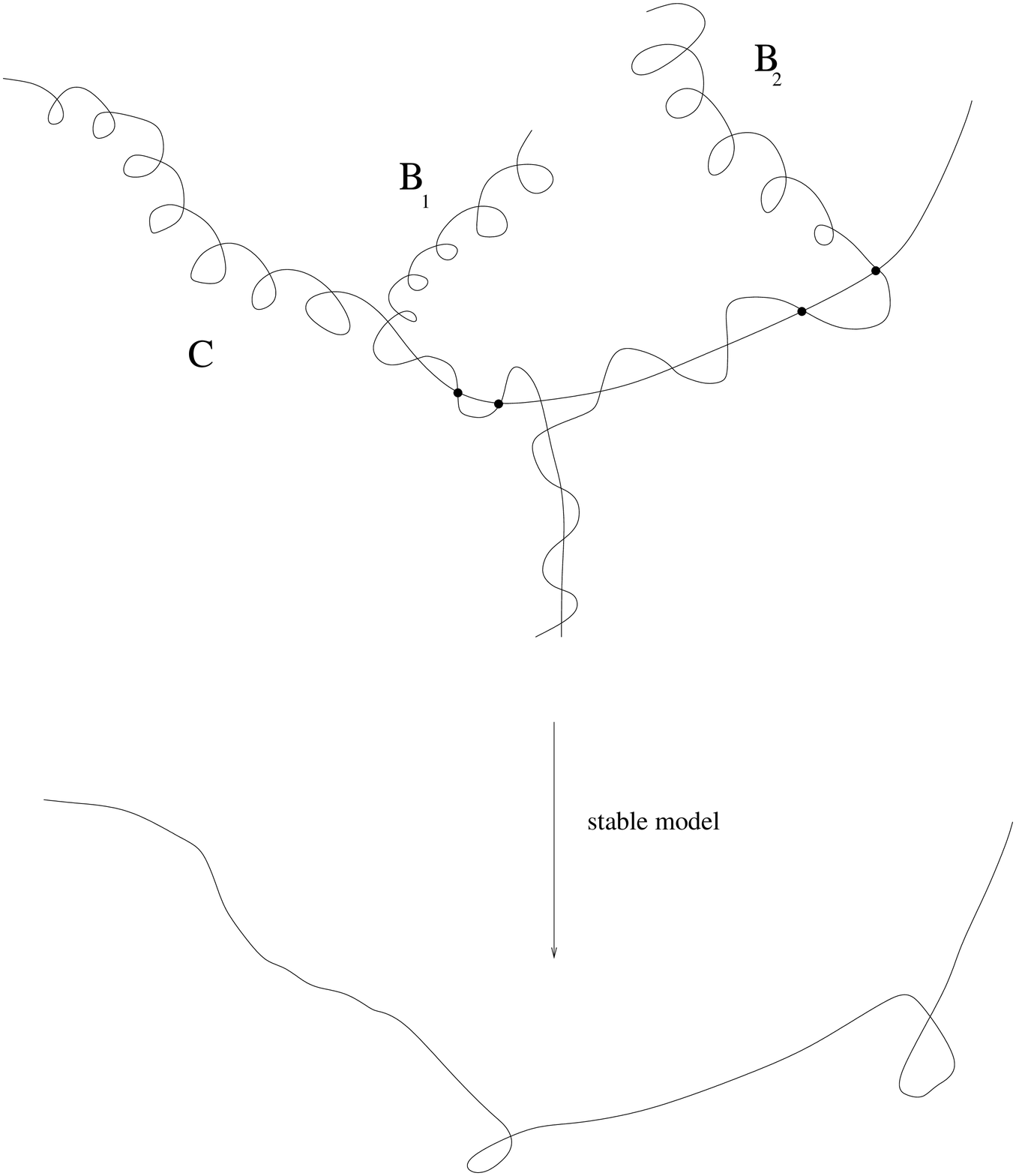}  
\]
\caption{Members of $V_{(\delta, s,a)}(S)$, $s=2$, with marked nodes the ones without dots, and the stable model.}
\label{fig:dis2}
\end{figure}

In this way we determine (at least) an irreducible, fully complete component $\V^*_{m, \zeta}$ of 
 $\overline \V_{m,\zeta}$  with its moduli map
\[
\psi^ *_{m,\zeta}: \V^*_{m, \zeta} \to \overline \M_{{g_{1,\delta}}+s}.
\]
(Note that $\V^*_{1,\delta}=\V^*_{\delta}$.)

\begin{proposition} \label{prop:dom} Let $p\geqslant 3$. 
Let $g$ and $s$ be positive integers with  $1  \leqslant g \leqslant 7$,  $g \leqslant p$ and $s\leqslant 4$. Set $\zeta=p(m)-g-s$. Then, for any integer $m \geqslant s+1$, the map $\psi^ *_{m,\zeta}$ is dominant.
\end{proposition}

\begin{proof} We keep the notation introduced above.  Take a general $\Gamma\in \M_g$ and let $(x_{i},y_{i})$, for $1\leqslant i \leqslant s$, be general pairs of points on $\Gamma$. By  Proposition \ref{prop:unica2}  and Lemma \ref{lemma:fibra} we may construct nodal curves 
\[\overline C=C + B_{1} + \cdots + B_{s} \in V_{(\delta,s,m-s)}(S), \,\,\, \text{with}\,\,\,  x_{i}+ y_{i} \leqslant \mathfrak g_i, \,\,\, \text{for}\,\,\, 1\leqslant i\leqslant s, \]
where  $\delta=p-g$,  $(S,C) \in \V^*_{1,\delta}=\V^*_{\delta}$  and  $\Gamma = \psi^*_{\delta}(C)$.  
The stable model $\overline \Gamma$  of the curve obtained by normalizing $\overline C$ at all the marked nodes, is the curve obtained by pairwise identifying the inverse images of the points $x_{i},y_{i}$ on $\Gamma$, for $1\leqslant i\leqslant s$. This is, by construction, a general member of the $s$--codimensional locus $\Delta_{s,g+s}$ of irreducible $s$--nodal curves in $\overline{\M}_{g+s}$, which therefore sits in the closure of the image of $\psi^ *_{m,\zeta}$. 

We claim that any component of $\V_{(\delta,s,m-s)}$ is a whole component of
$({\psi}^ *_{m,\zeta})^{-1}(\Delta_{s,g+s})$. Indeed, if it were not, then the general member of the component containing it would be a curve $C^ *$ specializing to $\overline C$, which means that the normalization of $C^ *$ at its marked nodes would consist of a component of arithmetic genus 
$g+s'$ with $s' \leqslant s$ nodes (giving rise, in the specialization, to $s'$ of the rational curves $B_i$ in $\overline C$), plus $s-s'$ rational curves (tending to the remaining curves $B_i$ in $\overline C$). In any event, $C^ *$ has geometric genus $g$. So the relative dimension over $\K_p$ of such a family is $g$, equal to the relative dimension of $\V_{(\delta,s,m-s)}$, a contradiction.

It follows that the fiber dimension over $\Delta_{s,g+s}$ is
\[ \dim (\V_{(\delta,s,m-s)}) - \dim (\Delta_{s,g+s}) =(g+19)-(3(g+s)-3-s)= 22-2(g+s),\]
hence the general fiber of $\psi^ *_{m,\zeta}$ has at most this dimension, so
\[\dim({\rm Im}(\psi^ *_{m,\zeta}))\geqslant  [g+s+19]-[22-2(g+s)]=3(g+s)-3\]
proving the assertion.
\end{proof}

\begin{cor} \label{cor:tutteledominanze}
  Assume $p \geqslant 3$.  With $g:=p(m)-\delta$, the map $\psi^ *_{m,\delta}$ is dominant if:\\
  \begin{inparaenum}
  \item [$\bullet$] $m=2$,  $2 \leqslant g \leqslant 8$  and $p \geqslant g-1$;\\
  \item [$\bullet$] $m=3$,  $ 2 \leqslant g \leqslant 9$  and $p \geqslant g-2$;\\
   \item [$\bullet$] $m=4$,  $ 2 \leqslant g \leqslant 10$  and $p \geqslant g-3$;\\
   \item [$\bullet$]$m \geqslant 5$,  $2 \leqslant g \leqslant 11$  and  $p \geqslant g-4$.
  \end{inparaenum}
\end{cor}

\begin{proof}
 For any $1 \leqslant s \leqslant 4$, Proposition \ref{prop:dom} ensures dominance of  $\psi^ *_{m,\delta}$ if $m \geqslant s+1$, $1+s \leqslant g \leqslant 7+s$ and $p \geqslant g-s$, proving the result.
\end{proof}

  Together with Remark \ref{rem:generibassi}, and the fact that the assertion is trivial for $g=0$, this finishes the proof of  Theorem \ref{thm:main}(A) for $m \geqslant 2$.

\section{Generic finiteness of the moduli map for $m \geqslant 2$}\label{sec:genfin}
In this section we will use Proposition \ref{prop:unica2} and Corollary \ref{cor:tutteledominanze} to prove generic finiteness of $\psi^ *_{m,\delta}$ for $m>1$ as in Theorem \ref {thm:main}.

The following lemma is trivial and the proof can be left to the reader:

\begin{lemma} \label{lemma:finitezzasu'}
  If there exists a component $\V \subseteq \V_{m,\delta}$ such that
${\psi_{m,\delta}}_{|\V}$ is generically finite onto its image, then for each component $\W \subseteq \V_{m,\delta-1}$ such that $\V$ is included in $\W$ (see Remark \ref {rem:filtr}), the map ${\psi_{m,\delta-1}}_{|\W}$ is generically finite onto its image. \end{lemma}

\begin{lemma} \label{lemma:baseperfinitezza}
  Assume that $\gamma$ and $\mu$ are integers such that $\psi_{\mu,p(\mu)-\gamma}$ is generically finite onto its image on a component $\V'$ of $\V_{\mu,p(\mu)-\gamma}$. Let $a \geqslant 1$ and $b \geqslant 0$ be integers such that  there is a component  $\V''$ of $\V_{a,p(a)-b}$ satisfying: \\
\begin{inparaenum}[(i)] 
\item  for general $(S,C)\in \V'$ and $(S,D)\in \V''$, $D$ intersects $C$ transversally; \\
\item for general $[S] \in \K_p$, the restriction 
${\psi_{a,p(a)-b}}_{|V''(S)}$ is generically finite, where $V''(S)$ denotes the fiber of $(\phi_{a,p(a)-b})_{|\V''}: \V'' \to \K_p$ over $[S]$.
  \end{inparaenum}\\
  Then for $m=\mu+a$ and $\delta=p(\mu+a)-(\gamma+b+1)$, the map
  $\psi_{m,\delta}$ is generically finite onto its image on some component $\V$ of $\V_{m,\delta}$. 
\end{lemma}

\begin{proof}
  Pick general elements in $\V'$ and $\V''$ like in (i), with $\widetilde{C}=\psi_{\mu,p(\mu)-\gamma}(C) \in \M_{\gamma}$ and
  $\widetilde{D}=\psi_{a,p(a)-b}(D) \in \M_{b}$. Then $(S,C + D)$ corresponds to a point of a component $\V$ of  ${\V_{m,\delta}}$, if we consider as marked nodes of $C+D$ all of its nodes but two in $C \cap D$. Let $\psi={\psi_{m,\delta}}_{|\V}$. This map sends $C+D$ to $\widetilde C$ plus $\widetilde D$ glued at two point (with a further contraction of $\widetilde D$ if $b=0$).

We denote by $\B \subset {\V}$ the subset of curves of the form $C' + D'$ where $C'$ and $D'$ are in $\V'$ and $\V''$, respectively. Then $\dim(\B)=19+\gamma+b$. 

Because of the hypotheses, the map $\psi_{|\B}$ is generically finite onto its image. Hence
$\dim( \im (\psi_{\B}))= 19+\gamma+b$. Since the general element of $\im ({\psi})$ is smooth, we must have $\dim (\im ({\psi})) \geqslant 20+\gamma+b$. Therefore, the general fiber of ${\psi}$ has dimension at most
\[ \dim ({\V}) - (20+\gamma+b)= [19+(\gamma+b+1)]-(20+\gamma+b)=0,\]
and the result follows.
\end{proof}

\begin{cor} \label{cor:tuttelefinitezze}
The map $\psi_{m,\delta}$ is generically finite on some component of $\V_{m,\delta}$ in the following cases, with $g:=p(m)-\delta$:\\
  \begin{inparaenum}[$\bullet$]
  \item $2 \leqslant m \leqslant 4$, $g \geqslant 16$ and $p \geqslant 15$;\\
  \item $m \geqslant 5$, $p \geqslant 7$  and $g \geqslant 11$.
  \end{inparaenum}
\end{cor}

\begin{proof}
Let $m \geqslant 2$. If $p \geqslant 15$, we apply Lemma \ref{lemma:baseperfinitezza} with $\mu=1$, 
$\gamma=15$, $a=m-1 \geqslant 1$, $b=0$, $\V'=\V^*_{p-15}$, the component for which we proved generic finiteness of  $\psi_{p-15}$ in 
Proposition \ref{prop:unica2} and $\V''=\U_a \sub \V_{a,p(a)}$ the component consisting of rational curves that degenerate to a curves of the type $U_a(R)$ as in Definition \ref{def:curvedichen},  cf. Proposition \ref{prop:curvedichen}.  Condition (ii) of Lemma \ref{lemma:baseperfinitezza} is  satisfied and an argument as in the proof of Lemma \ref{lemma:fibra} shows that also condition (i) is satisfied. Then Lemma \ref{lemma:baseperfinitezza} implies that $\psi_{m,\delta}$ is generically finite on a suitable component of the universal Severi variety
for $g=16$. Lemma \ref{lemma:finitezzasu'} yields that  $\psi_{m,\delta}$ is generically finite on some component
for all $g \geqslant 16$, as stated.

If $m \geqslant 5$, $p \geqslant 7$ and $g =11$, the map $\psi_{m,\delta}$ is generically finite 
on some component of $\V_{m,\delta}$ by Corollary \ref{cor:tutteledominanze}, hence Lemma \ref{lemma:finitezzasu'} yields the same for all $g \geqslant 11$. 
\end{proof}

 This proves Theorem \ref{thm:main}(B) for $m \geqslant 2$.

\end{document}